\documentclass[
12pt,a4paper,notitlepage]{scrartcl}

\usepackage[utf8]{inputenc}
\usepackage[T1]{fontenc}
\usepackage[english]{babel}
\usepackage{abstract}
\usepackage{titling}
\usepackage{amsmath}
\usepackage{amsfonts}
\usepackage{amssymb}
\usepackage{amsthm}
\usepackage{cite}
\usepackage{soul}
\usepackage{delarray}
\usepackage{wasysym}
\usepackage{graphics}
\usepackage{turnstile}
\usepackage{stmaryrd}
\usepackage{nicefrac}
\usepackage{faktor}
\usepackage[dots]{12many}
\usepackage{float}
\usepackage{yfonts}
\usepackage{wrapfig}

\usepackage{enumerate}
\usepackage[retainorgcmds]{IEEEtrantools}

\usepackage{hyperref}

\newcommand\eps{\varepsilon}
\newcommand\del{\partial}
\renewcommand\ge{\geqslant}
\renewcommand\le{\leqslant}
\newcommand\abs[1]{\lvert #1 \rvert}
\newcommand\norm[1]{\| #1 \|}

\newcommand\scp[2]{\ensuremath{\langle #1 , #2 \rangle}}
\newcommand\R{\mathbb{R}}
\newcommand\N{\mathbb{N}}

\newcommand\C{\mathbb{C}}
\newcommand\Z{\mathbb{Z}}
\newcommand\mO{\mathcal{O}}
\newcommand\mS{\mathcal{S}}

\newcommand{\dd}{\mathop{}\!\mathrm{d}}

\newcommand\ce{\mathrel{\mathop:}=}

\DeclareMathOperator{\Id}{Id}

\DeclareMathOperator{\Vol}{Vol}

\DeclareMathOperator{\RealPart}{Re}
\renewcommand\Re{\RealPart}

\edef\int{\int\limits}
\edef\sum{\sum\limits}
\edef\prod{\prod\limits}
\edef\inf{\inf\limits}
\edef\sup{\sup\limits}

\makeatletter \g@addto@macro\@floatboxreset\centering \makeatother

\numberwithin{equation}{section}

\makeatletter
\if@titlepage
   \renewenvironment{abstract}{%
       \titlepage
       \null\vfil
       \@beginparpenalty\@lowpenalty
       \par\medskip\noindent{\bfseries\abstractname.}
         \@endparpenalty\@M
       }%
      {\par\vfil\null\endtitlepage}
\else
   \renewenvironment{abstract}{%
       \if@twocolumn
         {\bfseries\abstractname.}%
       \else
         \quotation
         \small\noindent{\bfseries\abstractname.}
         \@endparpenalty\@M
       \fi}
       {\if@twocolumn\par\medskip\else\endquotation\fi}
\fi
\makeatother

    \makeatletter
    \renewcommand*{\@fnsymbol}[1]{\ensuremath{\ifcase#1\or \or \dagger\or \ddagger\or
        \mathsection\or \mathparagraph\or \|\or **\or \dagger\dagger
        \or \ddagger\ddagger \else\@ctrerr\fi}}
    \makeatother

\theoremstyle{plain}
\newtheorem{theopr}{Theorem}

\newtheorem{theo}{Theorem}
\newtheorem{lemma}{Lemma}[section]

\newtheorem{prop}[lemma]{Proposition}
\newtheorem*{rem}{Remark}

\pretitle{\begin{center}\Large\bfseries}
\posttitle{\par\end{center}\vskip 0.5em}
\preauthor{\begin{center}}
\postauthor{\end{center}}
\predate{\par\centering}
\postdate{\par}

\title{Dimension free $L^p$-bounds of maximal functions associated
  to products of Euclidean balls}
\author{Frederic Sommer\thanks{The author has been supported by
    the DFG grant DFG-Project MU 761/11-1}
}

\newcommand{\Addresses}{{
  \bigskip
  \footnotesize

  \textsc{Mathematisches Seminar, Christian-Albrechts-Universität zu Kiel,
    Ludewig-Meyn-Str. 4, 24118 Kiel, Germany}\par\nopagebreak
  \textit{E-mail address}: \texttt{sommer@math.uni-kiel.de}

}}

\begin{document}

\maketitle

\begin{abstract}
A few years ago, Bourgain proved that the centered Hardy-Littlewood
maximal function for the cube has dimension free $L^p$-bounds for
$p>1$. We extend his result to products of Euclidean balls of
different dimensions. In addition, we provide dimension free
$L^p$-bounds for the maximal function associated to products of
Euclidean spheres for $p > \frac{N}{N-1}$ and $N \ge 3$, where $N-1$
is the lowest occurring dimension of a single sphere. The
aforementioned result is obtained from the latter one by applying the
method of rotations from Stein’s pioneering work on the spherical
maximal function. 
\end{abstract}
\tableofcontents

\section{Introduction}
\label{part1}
For any convex body $B$ and any $f \in L^1_{\text{loc}}(\R^n)$,
denote the centered Hardy-Littlewood maximal function of $f$
associated to $B$ by
\begin{equation*}
  M_Bf(x) := \frac1{|B|} \sup\limits_{t>0} \int_{B} |f(x+ty)| \dd y.
\end{equation*}
The history of dimension free bounds starts with the celebrated result
by Stein \cite{ste}, who discovered that for $1 < p \le \infty$, the
centered Hardy-Littlewood maximal operator associated to the Euclidean
ball in $\R^n$ has an $L^p$-bound that does not depend on the
dimension $n$. This result has been obtainend by averaging over
spheres, and using the $L^p$-boundedness of the spherical maximal 
operator for $n >2$ and $p > \frac{n}{n-1}$ (see also
\cite{stestr}).\\
One can ask whether this holds for any convex body $B$, and several
more results have been obtained since then. Bourgain \cite{bou1}
proved that the centered maximal operator is $L^2$-bounded 
independently of $n$ for any convex, centrally symmetric body in
$\R^n$. To do this, he showed that there are a unique linear map $A
\colon \R^n \to \R^n$ with $\det A = 1$ and a constant $L=L(B)$ such
that for every $\xi \in S^{n-1}$, we have
\begin{equation}
  \label{eq:isotropy}
  \int_{A(B)} |\scp x \xi|^2 \dd x = L(B)^2.
\end{equation}
If $B = A(B)$, we say that $B$ is in isotropic position, which
can always be assumed since $\| M_B \|_{p \to p} = \| M_{A(B)} \|_{p
  \to p}$. The $L^2$ result then follows from exploiting the decay of
the Fourier transform of $\chi_B$, also proven in
\cite{bou1}. Precisely, there is a universal constant $C$ such that
for every $B$ in isotropic position, we have the following estimates.
\begin{equation}
  \label{eq:bourgaindecay}
  |\widehat{\chi_B}(\xi)| \le C |L \xi|^{-1}, \quad |1 - \widehat{
    \chi_B}(\xi)| \le C | L \xi |, \quad |\scp{\xi}{\nabla
    \widehat{\chi_B}(\xi)}| \le C
\end{equation}
The $L^2$ result has been improved to the range
$p>\nicefrac32$ in \cite{bou3}, and independently by Carbery
\cite{car}. Going further, M\"uller \cite{mue} showed that for $p>1$,
we have $\| M_B \|_{p \to p} \le C(p, \sigma, Q)$, with some geometric
invariants $\sigma=\sigma(B)$ and $Q=Q(B)$. For isotropic $B$, these
are defined as 
\begin{equation*}
  \sigma(B)^{-1} = \max \{ \Vol_{n-1}(\{ x \in B : \scp x \xi = 0\}) :
  \xi \in S^{n-1} \},
\end{equation*}
\begin{equation*}
  Q(B) = \max \{ \Vol_{n-1}(\pi_\xi(B)) : \xi \in S^{n-1} \},
\end{equation*}
where $\pi_\xi$ is the projection onto the subspace orthogonal to
$\xi$. M\"uller also showed that if $B$ is an $\ell^q$-ball for some
$1 \le q < \infty $, then $\sigma(B)$ and $Q(B)$ do not depend on $n$,
bounding the corresponding maximal function independently of the
dimension. However, in the case $q=\infty $,
i.e. $B = \mathopen[ -\nicefrac12, \nicefrac12 \mathclose]^n$, we have
$Q(B) = \sqrt n$, and the problem remained open until
recently. However, in \cite{bou2}, Bourgain succeeded to show that
also in the case of the cube, there exists a dimension free bound. A
survey of all these results, with attention to further details, has
recently been published by Deleaval, Gu{\'e}don, and Maurey
\cite{dgm}. Following the latest result by Bougain, the purpose of
this work is to prove the following theorem. 
\begin{theo}
  \label{theo:one}
  Let $B \ce B_1 \times \cdots \times B_\ell$ be a direct product of
  $\ell \ge 1$ Euclidean balls $B_k$ in $\R^{n_k}$, $n_k \ge 1$, and
  put $n \ce n_1 + \ldots + n_\ell$. Assume further that $1 < p \le
  \infty $. Then 
  \begin{equation}
  \label{eq:theone}
    \norm{M_B f}_p \le C_p \norm{f}_p
  \end{equation}
  for every $f \in L^p(\R^n)$, where $C_p$ depends only on $p$, but
  not on $\ell$ or the dimensions $n_k$ of the factors.
\end{theo}
We will see that, similarly to the case of the cube, the invariant
$Q(B)$ from \cite{mue} grows like $\sqrt \ell$, so that we cannot make
use of M{\"u}ller's bounds.\\
An outline of our approach goes as follows. A na{\"i}ve way would be
to estimate 
\begin{IEEEeqnarray*}{rCl}
  M_Bf(x) &\leq& \frac1{|B|} \sup_{t_1, \ldots, t_\ell>0}
  \int_{B_\ell} 
  \cdots \int_{B_1} |f(x + (t_1y^{(1)}, \ldots, t_\ell y^{(\ell)}))|
  \dd y^{(1)} \ldots \dd y^{(\ell)} \\
  &\le& \frac1{|B_\ell|} \sup_{t_\ell > 0} \int_{B_\ell} \cdots
  \frac1{|B_1|} \sup_{t_1 > 0} \int_{B_1} |f(x + (t_1y^{(1)}, \ldots,
  t_\ell y^{(\ell)}))| \dd y^{(1)} \ldots \dd y^{(\ell)},
\end{IEEEeqnarray*}
getting iterated maximal functions, where we write $x = (x^{(1)},
\ldots, x^{(\ell)})$ with $x^{(k)} \in \R^{n_k}$. Due to Stein's
dimension free bound for the Euclidean ball, we can estimate each
iterated maximal function to get 
\begin{equation}
  \label{eq:uselessbound}
  \| M_Bf \|_p \le C_p^\ell \| f \|_p
\end{equation}
for $1 < p < \infty$. Let $B^{(N)}$ be the Euclidean ball in $\R^N$
with Lebesgue measure $1$. Since the Fourier transform satisfies
$\chi_{B^{(N)}} = \mO(|\xi|^{-\frac{N+1}2})$ as $|\xi| \to \infty$,
we see that the Fourier transform of $\chi_B$ has a decay of at least
this rate (with $N = \min\limits_{1 \le k \le \ell} n_k$) on certain
subspaces of $\R^n$, while the decay is even better elsewhere,
behaving in a way very similar to the cube. Hence, in order to show
Theorem \ref{theo:one}, we shall make use of some of the central
arguments in Bourgain's approach for the cube to attain a bound that
depends on $p$ and $\max\limits_{1 \le k \le \ell} n_k$, but not on
$\ell$. From here on, we aim to combine this with
\eqref{eq:uselessbound} to achieve a bound as in Theorem
\ref{theo:one}. For this, we will provide a similar theorem for
spheres. Let $S \ce S_1 \times \cdots \times S_\ell$ be a product of
Euclidean spheres $S_k$ in $\R^{n_k}$. Let $\sigma_S$ be the product
of the spherical Lebesgue measures of the $S_k$ and define the maximal operator
$M_S$ on $\mS(\R^n)$ by 
\begin{equation*}
  M_S f(x) \ce \frac1{|S|} 
  \sup\limits_{t>0} \int_{S} |f(x + t \omega)| \dd \sigma_S(\omega).
\end{equation*}
Here, $|S|$ denotes the $(n - \ell)$-dimensional volume of $S$.
\begin{theo}
  \label{theo:two}
  Let $S$ be as above and $n_k \ge 3$ for each $k$. Put $N =
  \min\limits_{1 \le k \le \ell} n_k$ and assume that $p > \frac
  N{N-1}$. Then we have
  \begin{equation}
  \label{eq:thetwo}
    \norm{M_S f}_p \le C_p \norm{f}_p
  \end{equation}
  for every $f \in \mS(\R^n)$, where $C_p$ only depends on $p$.
\end{theo}
As in the well-known case of $\ell=1$, the lower bound for $p$ in
Theorem \ref{theo:two} is optimal. We will prove Theorem
\ref{theo:two} by using Stein's approach to see that we can increase
the dimension of each factor of $S$ without increasing $C_p$. With
Stein's argument, we also get $\|M_B f\|_p \le \|M_S f\|_p$ if $B$ is
the convex hull of $S$, hence Theorem \ref{theo:two} is sufficient for
$N$ large enough, depending on a fixed value of $p$. To show Theorem
\ref{theo:two}, we proceed with applying an interpolation similar as
in Carbery's proof for $p > \nicefrac32$, in which he makes use of the
fact that $\scp{\xi}{\nabla \widehat {\chi_B}(\xi)}$ has bounded
$L^2$-multiplier norm for a general convex body $B$. In his final
remark, he states that if this derivative has bounded $L^q$-multiplier
norm for a bigger range of $q$, we would obtain a better bound on $p$
than $\nicefrac32$, which is the case if we can bound the higher
fractional derivatives 
\begin{equation*}
  \bigg(\frac{\dd}{\dd r}\bigg)^z \widehat{ \sigma_{S^{N-1}}}(r\xi)
\end{equation*}
of the Fourier transforms of spherical measures, where
$\Re(z)=\frac{N+1}2$. Bounding these uses several ideas from Müller's
proof, where he proceeded similarly for arbitrarily higher
derivatives to bound them in terms of certain geometric
invariants. Also, a lot of the calculations will rely on the explicit
forms and the decay of $\widehat{ \sigma_{S^{N-1}}}$ and $\widehat
{\sigma_{S}}$.\\
Hence, if we fix $p > 1$, we have attained a dimension free bound, for
$N$ large enough, by generalizing Stein's and Carbery's ideas. The
remaining cases are covered by our generalization of Bourgain's
arguments for the cube, achieving a bound only depending on the
finitely many remaining $N$ and thus only on $p$.\\\\
Throughout this paper, we write $B_r^N$ for the $N$-dimensional
Euclidean ball of radius $r$, and $B^N$ if $r=1$. We put $r_N \ce
\pi^{-\nicefrac12} \Gamma(\frac N2+1)^{\nicefrac1N}$. Similarly,
$S^{N-1}_r$ denotes the $(N-1)$-dimensional sphere of radius $r$ in
$\R^N$, with $S^{N-1}$ being the sphere of radius $1$. The Fourier
transform of $f \in L^1(\R^n)$ is written as
\begin{equation*}
  \hat f(\xi) = \int_{\R^n} e^{-2\pi i \scp x \xi} f(x) \dd x,
\end{equation*}
and it will also be used in the distributional sense.
\section{Independence of the number of factors}
\label{part2}
This part is mainly a walkthrough of \cite{bou2}, where we will
omit any proof that does not need any further modification.\\
Let $B$, $\ell$, and $n$ be as in Theorem \ref{theo:one} and let
$m(\xi) = \widehat{\chi_B}(\xi)$. We group the variables by setting 
\begin{equation*}
  V_k \ce \bigg\{ \sum_{j=1}^{k-1} n_j + 1, \ldots, \sum_{j=1}^k n_j
  \bigg\}, \quad k \in \ito \ell.
\end{equation*}
The goal of this section is to show the following weaker result.
\begin{prop}
  \label{prop:bouN}
  Let $N \ce \max\limits_{1 \le k \le \ell} n_k$. Then
  \begin{equation}
    \label{eq:propboun}
    \| M_B f \|_p \le C_{p, N} \| f \|_p
  \end{equation}
  for every $f \in L^p(\R^n)$, $1 < p \le \infty$.
\end{prop}
It is enough to consider the case $B = (B_{r_N}^N)^\ell$: Since $\| M_B
\|_{p \to p}$ is invariant under linear 
transformations of $B$ (as already mentioned in \cite{bou1}), we can
assume that $|B_k| = 1$ for every $k \in \ito \ell $ (hence $B_k =
B_{r_{n_k}}^{n_k}$). By change of coordinates, we can assume $B =
\prod_{j=1}^N (B_{r_j}^{j})^{\ell_j}$ for certain $\ell_j \in \N$, where,
without loss of generality, we allow that $\ell_j=0$. Suppose that we
already found constants $C_{p, j}$ independent of $\ell_j$ such that
\begin{equation*}
  \| M_{(B_{r_j}^{j})^{\ell_j}} f \|_p \le C_{p, j} \| f \|_p.
\end{equation*}
Then we can argue similarly as in \eqref{eq:uselessbound} to get
\begin{equation*}
  \| M_B f \|_p \le \prod_{j=1}^\ell C_{p, j} \|f\|_p.
\end{equation*}
Let $B = (B_{r_N}^N)^\ell$, i.e. $n = N\ell$. Then $B$ is in isotropic
position, with 
\begin{IEEEeqnarray*}{rCl}
  L(B)^2
  = \int_B |\scp x \xi|^2 \dd x
  &=& \sum_{k=1}^\ell \int_B |\scp{x^{(k)}}{\xi^{(k)}}|^2 \dd x +
  \underbrace{\sum_{k, k'=1}^\ell \int_B \scp{x^{(k)}}{\xi^{(k)}}
  \scp{x^{(k')}}{\xi^{(k')}} \dd x}_{=0} \\
  &=& \sum_{k=1}^\ell \int_{B_{r_N}^N} |\scp{y}{\xi^{(k)}}|^2 \dd y \\
  &=& L(B_{r_N}^N)^2
\end{IEEEeqnarray*}
for every $\xi \in S^{n-1}$, since $|B_{r_N}^N|=1$. Of course, any
Euclidean ball itself is in isotropic position due to rotation
invariance. We will first show that M{\"u}ller's bounds won't apply to
even this special case of our situation.
\begin{lemma}
  We have
  \begin{equation*}
    Q(B) = \sqrt \ell \cdot \pi^{-\nicefrac12} \frac{\Gamma(\nicefrac
      N2 + 1)}{\Gamma(\frac{N+1}2)} .
  \end{equation*}
\end{lemma}
\begin{proof}
  Let $\xi \in S^{n-1}$. We will estimate $|\pi_\xi(B)|$, the
  $(n-1)$-dimensional volume of the orthogonal projection of $B$ onto
  $\xi^\perp$. The geometric arguments from \cite{bf} and a limiting
  argument show that
  \begin{equation*}
    |\pi_\xi(B)| = \int_{\del B} (\scp{n(x)}{\xi})_+ \dd \sigma(x),
  \end{equation*}
  where $\sigma$ is the Lebesgue surface measure of $\del B$ and
  $n(x)$ is the corresponding normal vector (well-defined
  $\sigma$-almost everywhere on $\del B$). Write
  \begin{equation*}
    \del B = \bigcup_{k=1}^\ell H_k, \quad H_k = (B^N_{r_N})^{k-1}
    \times S^{N-1}_{r_N} \times (B^N_{r_N})^{\ell-k-1}.
  \end{equation*}
  Then the $H_k$ are the pendants to the faces of a cube, and by
  Fubini and rotation invariance of the spherical measure, we get
  \begin{IEEEeqnarray*}{rCl}
    |\pi_\xi(B)|
    &=& \sum_{k=1}^\ell \int_{H_k} (\scp{n(x)}{\xi})_+ \dd \sigma(x) \\
    &=& \sum_{k=1}^\ell \int_{S^{N-1}_{r_N}} (\scp{\omega}{\xi^{(k)}})_+
    \dd \sigma_{S^{N-1}_{r_N}}(\omega) \\
    &=& \sum_{k=1}^\ell |\xi^{(k)}| \int_{\omega_1 > 0} \omega_1 \dd
    \sigma_{S^{N-1}_{r_N}}(\omega) \\
    &=& r_N |B^{N-1}_{r_N}| \sum_{k=1}^\ell |\xi^{(k)}|.
  \end{IEEEeqnarray*}
  We can maximize $|\pi_\xi(B)|$ by choosing $\xi^{(k)} =
  \ell^{-\nicefrac12} e_1 \in \R^N$. Since
  \begin{equation*}
    r_N |B^{N-1}_{r_N}| = r_N^N |B^{N-1}| = \pi^{-\nicefrac12}
    \frac{\Gamma(\nicefrac N2 + 1)}{\Gamma(\frac{N+1}2)},
  \end{equation*}
  the lemma follows.
\end{proof}
Consider the same decomposition as
in Bourgain's proof. Let $H \colon \R^n \to \R$ such that $\hat H(\xi)
= e^{-|\xi|^2}$ and put $\Omega^{(s)} = \chi_B \ast (H_{2^s} -
H_{2^{-s+1}})$ for $s \ge 1$. Then
\begin{equation*}
  \chi_B = (\chi_B \ast H) + \sum_{s=1}^\infty \Omega^{(s)}.
\end{equation*}
Using only the
well-known estimates $|\xi||m(\xi)| < CL(B)^{-1} < C'$ and
$|\scp{\xi}{\nabla m(\xi)}| < C$ for general convex
bodies, Lemma 3  from \cite{bou1} and the exponential decay of $H$
give us (see (1.16)
in \cite{bou2}) 
\begin{equation*}
  \big\|\sup_{t>0} |f\ast (\Omega^{(s)})_t| \big\|_2 < C
  2^{-\nicefrac s2}\|f\|_2
\end{equation*}
for $s \ge 1$ and
\begin{equation*}
  \big\|\sup_{t>0} |f\ast (H \ast \chi_B)_t| \big\|_2 < C \|f\|_2.
\end{equation*}
For $1 < p < 2$, it suffices to find a bound
\begin{equation*}
  \big\| \sup_{t>0} | f \ast ( \chi_B \ast H_{2^{-s}} )_t | \big\|_p \le
  C_{p,s} \|f\|_p,
\end{equation*}
$s \in \N$, so that $C_{p,s}$ is suitable for interpolation with the
$L^2$-estimates. For this, Bourgain takes the ideas from \cite{mue} to
conclude that it suffices to find an $L^p$-bound for the operator $T$,
defined by
\begin{equation}
  \label{eq:lem2op}
  \widehat{Tf}(\xi) = \abs\xi m(\xi) e^{-4^{-s}|\xi|^2} \hat f(\xi),
\end{equation}
and to estimate
\begin{equation}
  \label{eq:l2bounds2}
  \sup_{\abs\xi = 1} \int_{\R^n} | \scp x \xi |^k (\chi_B \ast
  H_{2^{-s}}) (x) \dd x < C_k, \quad k\ge 1.
\end{equation}
For \eqref{eq:l2bounds2}, we can make use of the fact that $B$ is
symmetric in each coordinate, applying Khinchin's inequality as in
\cite[p.~279]{bou2}.\\
To estimate the operator $T$ in \eqref{eq:lem2op}, Bourgain uses a
duality argument as in \cite[p.~306]{mue} and Stein's dimension free bound on
the Riesz transforms (see \cite{ste}), which leaves him with proving
Lemma 3 from \cite{bou2}. In our situation, Proposition
\ref{prop:bouN} follows from the following, similar Lemma.
\begin{lemma}
  \label{lemma:lemma3}
  Let $N \ce \max\limits_{1 \le k \le \ell} n_k$. For $R\ge 2$ and $j
  \in \ito n$, let $\mu_j = \del_j (\chi_B \ast
  H_{\nicefrac1R})$. Then for every $2 \le p < \infty$, $0 < \eps <
  1$, and $f \in L^p$ we have 
  \begin{equation}
    \label{eq:lemma3}
    \bigg\| \bigg( \sum_{j=1}^n | f \ast \mu_j |^2
    \bigg)^{\nicefrac12} \bigg\|_p \le C_{p,\eps, N} R^{12N \cdot
      \eps} \| f \|_p, 
  \end{equation}
  with $C_{p,\eps, N}$ independent of $R$ and $\ell$.
\end{lemma}
Fix $2 \le p < \infty $, $R\ge 2$, and $0 < \eps < 1$. The direct
interpolation from \cite[p.~280]{bou2} shows that
\begin{equation}
  \label{eq:mujp}
  \bigg\| \bigg(\sum_{j=1}^n |f \ast \mu_j|^2 \bigg)^{\nicefrac12}
  \bigg\|_p
  \le C_p R^{1-\nicefrac2p} \| f \|_p.
\end{equation}
 We proceed
with Bourgain's Fourier localization. By means of Pisier's result
on contractive semigroups \cite[p.~390]{pis}, we get the following.
\begin{lemma}
  \label{lemma:lemma5}
  Let $\eta = (1 - |x|)_+$, $t>0$ and let $T_j\colon
  L^p(\R^n) \to L^p(\R^n)$ be the convolution by $\eta_t$ in the $j$-th
  variable. For every $k \in \{1,\ldots,\ell\}$ let
  \begin{equation*}
    S_k \ce \prod_{j \in V_k} T_j.
  \end{equation*}
  Furthermore, put for every $k \in \{0,\ldots,\ell\}$
  \begin{equation}
    \label{eq:holo2}
    A_k \ce \sum_{\substack{ A\subset \ito \ell\\ |A|=k}} \prod_{j \notin
      A} S_j \prod_{j \in A} (\Id - S_j).
  \end{equation}
  Then for $1 < q < \infty$, $\| A_k \|_{q\to q} \le C_q^k$ with $C_q$
  independent of $k$.
\end{lemma}
Fix $t \ce R^{-\eps}$ and let $A_k$ be as in \eqref{eq:holo2}. Then
$A_k \ge 0$ and $\sum_{k=0}^\ell A_k = \Id $. For some $K \ge 1 $ to
be chosen later, we decompose $f \in L^p$ as
\begin{equation*}
  f = \sum_{k=0}^K A_kf + g.
\end{equation*}
To achieve a good $L^2$-estimate on $\big( \sum_{j=1}^n | g \ast
\mu_j |^2 \big)^{\nicefrac12}$, we poof a variant of Lemma 6 in
\cite{bou2}, where we write $\xi = (\zeta_1, \ldots, \zeta_\ell)$
with each $\zeta_k \in \R^N$.
\begin{lemma}
  \label{lemma:lemma6}
  For every $\delta>0$ and $k \ge 1$, we have 
  \begin{equation}
    \label{eq:lemma6}
    |m(\xi)| \le C_{k, N} \bigg( 1 + \sum_{|\zeta_j| \le
      R^\delta} |\zeta_j|^2 \bigg)^{-\nicefrac k2} R^{\delta k}. 
  \end{equation}
\end{lemma}
\begin{proof}
  To adapt the original proof to our setting, we need to show
  \begin{equation}
    \label{eq:l61}
    | \widehat{ \chi_{B_{r_N}^N}}(\zeta)| \le e^{-c|\zeta|^2}
  \end{equation}
  for $|\zeta| \le 1$ and
  \begin{equation}
    \label{eq:l62}
    | \widehat {\chi_{B_{r_N}^N}}(\zeta)| \le C = e^{-c'}
  \end{equation}
  for $|\zeta| > 1$, with $c' > 0$, $C < 1$, $\zeta \in \R^N$.
  Since
  \begin{equation*}
    \widehat {\chi_{B_{r_N}^N}}(\zeta)
    = r_N^{\nicefrac N2} |\zeta|^{-\nicefrac N2} J_{\frac N2}( 2
    \pi r_N |\zeta|)
  \end{equation*}
  with $J_\nu$ being the Bessel function of order $\nu$, we use the
  well-known series expansion
  \begin{equation}
    \label{eq:bessel}
    J_\nu(x) = \pi^{ -\nicefrac12} \frac{x^\nu}{2^\nu}
    \sum_{k=0}^\infty (-1)^k \frac{\Gamma(k + \nicefrac12)}{\Gamma(k + \nu
      + 1)} \frac{x^{2k}}{(2k)!}
  \end{equation}
  to get
  \begin{IEEEeqnarray*}{rCl}
    \widehat {\chi_{B_{r_N}^N}}(\zeta) 
    &=& \pi^{\frac{N-1}2} r_N^N \sum_{k=0}^\infty (-1)^k
    \frac{\Gamma(k + \nicefrac12)}{\Gamma(k + \nicefrac N2 + 1)}
    \frac{(2\pi r_N |\zeta|)^{2k}}{(2k)!} \\
    &=& \pi^{-\nicefrac12} \Gamma(\nicefrac N2+1) \sum_{k=0}^\infty
    (-1)^k \frac{\Gamma(k + \nicefrac12)}{\Gamma(k + \nicefrac N2 +
      1)} \frac{(2\pi r_N |\zeta|)^{2k}}{(2k)!} \\
    &=& \sum_{k=0}^\infty (-1)^k \frac{(2\pi r_N |\zeta|)^{2k}}{(2k)!}
    \prod_{j=0}^{k-1} \frac{j + \nicefrac12}{j + \nicefrac N2}.
    \yesnumber \label{eq:zusatzding}
  \end{IEEEeqnarray*}
  Let $f \colon \R \to \R$,
  \begin{equation*}
    f(t) = \sum_{k=0}^\infty (-1)^k \frac{(2\pi r_N t)^{2k}}{(2k)!}
    \prod_{j=0}^{k-1} \frac{j + \nicefrac12}{j + \nicefrac N2}.
  \end{equation*}
  Then we have
  \begin{equation*}
    f(0) = 1, \quad \frac{\dd}{\dd t} (f(t) - e^{-c t^2}) \bigg\vert_{t=0} = 0,
  \end{equation*}
  and
  \begin{equation}
    \label{eq:2ndderis}
    \frac{\dd^2}{\dd t^2} (f(t) - e^{-c t^2}) \bigg\vert_{t=0} = - 2 \pi r_N
    \frac1N + 2c
  \end{equation}
  for each $c>0$. Hence $f(t) - e^{-c t^2}$ has a local minimum at
  $t=0$ for $c > \frac{\pi r_N}N$, implying that \eqref{eq:l61}
  holds near $0$ with such a choice of $c$. Thus we are left with
  showing that $|\widehat {\chi_{B_{r_N}^N}}(\zeta)| < 1$ if $\zeta
  \neq 0$, which implies $|\widehat {\chi_{B_{r_N}^N}}(\zeta)| \le
  C_a < 1$ on $\R_{\ge a}$ for each $a > 0$. For that, we use that
  \begin{IEEEeqnarray*}{rCl}
    \widehat {\chi_{B_{r_N}^N}}(\zeta)
    &=& \frac{\Gamma(\frac N2 +
      1)}{\Gamma(\frac{N+1}2)\pi^{\nicefrac12}} 
    \int_{-1}^1 e^{it \cdot 2\pi r_N |\zeta|} (1-t^2)^{\frac{N-1}2}
    \dd t \\
    &=& \frac{\Gamma(\frac N2 +
      1)}{\Gamma(\frac{N+1}2)\pi^{\nicefrac12}} \cdot 
    2 \int_0^1 \cos(t \cdot 2\pi r_N |\zeta|) (1-t^2)^{\frac{N-1}2}
    \dd t. \yesnumber \label{eq:vdcint}
  \end{IEEEeqnarray*}
  Since the integrand in \eqref{eq:vdcint} is continuous and $|\cos(t
    \cdot 2\pi r_N |\zeta|)| < 1$ for all but finitely many $t \in
  \mathopen[ 0, 1 \mathclose]$ (if $|\zeta| \neq 0$), we get
  \begin{equation*}
    |\widehat {\chi_{B_{r_N}^N}}(\zeta)| 
    < \frac{\Gamma(\frac N2 +
      1)}{\Gamma(\frac{N+1}2)\pi^{\nicefrac12}} \cdot 2 \int_0^1
    (1-t^2)^{\frac{N-1}2} \dd t
    = 1.
  \end{equation*}
  Thus \eqref{eq:l61} and \eqref{eq:l62} hold, allowing us to conclude
  the Lemma as in \cite{bou2}. We will recall the argument.\\
  Let $I_0 = \{j \in \ito \ell : |\zeta_j|>1 \}$. Since for $j \notin I_0$,
  we have $|\zeta_j| \le 1$ and thus $|m(\zeta_j)| <
  e^{-c|\zeta_j|^2}$ by \eqref{eq:l61}, we get
  \begin{equation*}
    \prod_{j\notin I_0} |m(\zeta_j)|
    \le \exp\bigg(-C\sum_{j\notin I_0} |\zeta_j|^2 \bigg),
  \end{equation*}
  and also, by \eqref{eq:l62}
  \begin{equation*}
    \prod_{j\in I_0} |m(\zeta_j)| \le e^{-c'|I_0|}.
  \end{equation*}
  Together with the obvious estimate
  \begin{equation*}
    \sum_{|\xi_j|\le R^\delta} |\zeta_j|^2
    \le R^{2\delta}|I_0| + \sum_{j\notin I_0} |\zeta_j|^2 
    \le R^{2\delta} \bigg(|I_0| + \sum_{j\notin I_0} |\zeta_j|^2
    \bigg),
  \end{equation*}
  this leads to
  \begin{equation}
    \label{eq:mloc}
    |m(\xi)|
    \le \exp \bigg(-c R^{-2\delta} \sum_{|\zeta_j|\le R^\delta}
      |\zeta_j|^2 \bigg).
  \end{equation}
  But $e^{-C|x|} = \mO((1+|x|)^{-\nicefrac k2})$ for every $k \in
  \N_{>0}$, and hence, \eqref{eq:mloc} implies
  \begin{equation*}
    |m(\xi)|
    \le C_k \bigg( 1 + R^{-2\delta}\sum_{|\zeta_j|\le R^\delta} |\zeta_j|^2
    \bigg)^{-\nicefrac k2} 
    \le C_k \bigg( R^{-2\delta} + R^{-2\delta} \sum_{|\zeta_j|\le
      R^\delta} |\zeta_j|^2 \bigg)^{-\nicefrac k2},
  \end{equation*}
  immediately concluding the Lemma.
\end{proof}
With Lemma \ref{lemma:lemma6} we can establish a bound
\begin{equation}
  \label{eq:g2}
  \bigg\| \bigg( \sum_{j=1}^n | g \ast \mu_j |^2 \bigg)^{\nicefrac12}
  \bigg\|_2
  \le C_K R^{1-\frac{\eps K}{10}} \| f \|_2,
\end{equation}
with $C_K$ only depending on $K$. To achieve \eqref{eq:g2}, one
simply has to replace $\xi_j$ by $\zeta_j$ and $\hat \eta(t
\xi_j)$ by $\hat \eta(t\zeta_{j,1}) \cdots \hat \eta(t\zeta_{j,N})$ in
the corresponding proofs from \cite{bou2}. By interpolation with
\eqref{eq:mujp}, the choice $K = \big\lceil \frac{10(p-1)}\eps
\big\rceil$ gives us
\begin{equation*}
  \bigg\| \bigg(\sum_{j=1}^n |g \ast \mu_j|^2 \bigg)^{\nicefrac12}
  \bigg\|_p
  \le C_{p, \eps} \| f \|_p.
\end{equation*}
Hence it only remains to estimate
\begin{equation}
  \label{eq:akp}
  \bigg\| \bigg(\sum_{j=1}^n |A_k f \ast \mu_j|^2 \bigg)^{\nicefrac12}
  \bigg\|_p
  \le C_{p, \eps} R^{12N \cdot\eps} \| f \|_p
\end{equation}
for $0 \le k \le K$. For $S \subset \ito n$, put
\begin{equation*}
  \Gamma_S \ce \prod_{j\notin S} S_j \prod_{j\in S} (\Id - S_j). 
\end{equation*}
Then, by the triangle inequality, we have
\begin{IEEEeqnarray}{rCl}
  \bigg(\sum_{j=1}^n |A_k f \ast \mu_j|^2 \bigg)^{\nicefrac12}
  &\le& \bigg(\sum_{j=1}^n \bigg| \sum_{\substack{|S|=k\\
        j\notin S}} \Gamma_S f \ast \mu_j \bigg|^2
    \bigg)^{\nicefrac12} \label{eq:decouple1} \\
  && +\> \bigg(\sum_{j=1}^n \bigg| \sum_{\substack{|S|=k\\ j\in S}}
  \Gamma_S f \ast \mu_j \bigg|^2
  \bigg)^{\nicefrac12}. \IEEEeqnarraynumspace
  \label{eq:decouple2} 
\end{IEEEeqnarray}
Bourgain applies a stochastic method to
decouple the variables, which reduces \eqref{eq:decouple1} to the case
$k=0$ and \eqref{eq:decouple2} to the case $k=1$. We can acquire the
same by replacing $T_j$ by $S_j$ in that procedure. With that, we only
have to find suitable constants $b_0=b_0(R)$, $b_1=b_1(R)$ such that
\begin{equation}
  \label{eq:b00}
  \bigg\| \bigg( \sum_{j=1}^n | A_0 f \ast \mu_j |^2
  \bigg)^{\nicefrac12} \bigg\|_p
  \le b_0 \| f \|_p
\end{equation}
and
\begin{equation}
  \label{eq:b11}
  \bigg\| \bigg(\sum_{k=1}^\ell \big| \Gamma_k G_k f \big|^2
  \bigg)^{\nicefrac12} \bigg\|_p 
  \le b_1 \| f \|_p,
\end{equation}
where $\Gamma_k = (\Id-S_k)\prod\limits_{j\neq k} S_j$ and
\begin{equation*}
  G_kf = \bigg( \sum_{j \in V_k} | f \ast \mu_j |^2 \bigg)^{\nicefrac12}.
\end{equation*}
Let $B_p=B_{p,R}$ be minimal such that 
\begin{equation}
  \label{eq:ap}
  \bigg\| \bigg( \sum_{j=1}^n | f \ast \mu_j |^2 \bigg)^{\nicefrac12}
  \bigg\|_p
  \le B_p \|f \|_p.
\end{equation}
To estimate $b_1$, we need to rely on Lemmas 7-9 in \cite{bou2}. The
proofs of these will become more complicated in our setting, with
estimates that will depend on $N$. For the proofs, we will provide
slightly more details than in \cite{bou2}. Instead of using the
properties of the convolution operators $T_j$, we need to convolve
with a function that is roughly stable under small
translations. Bourgain considers the function
\begin{equation*}
  \varphi(x) \ce \frac c{1+x^4}
\end{equation*}
with $c$ so that $\int_\R \varphi(x) \dd x = 1$. Then
$C^{-1}\varphi(x-y) \le \varphi(x) \le C\varphi(x-y)$ for every $x \in
\R$ and $|y| \le 1$, $\varphi(x) =\mO(e^{-C|x|})$ as $|x| \to
\infty$, and $|1 - \hat \varphi(x)| < Cx^2$ for every $x \in \R$.\\
Fix $t_0 \ce R^{-3\eps}$, and let $\widetilde L_j \colon L^p(\R^n) \to
L^p(\R^n)$ be the convolution by $\varphi_{t_0}$ in the $j$-th
variable, $j \in \ito n$. For $k \in \ito \ell$, let
\begin{equation*}
  L_k \ce \prod_{j \in V_k} \widetilde L_j \quad \text{and} \quad L^{(k)}
  \ce \prod_{j \notin V_k} \widetilde L_j = \prod_{k' \neq k} L_{k'}.
\end{equation*}
With these notions, we show the following version of Lemma 7 from
\cite{bou2}.
\begin{lemma}
  \label{lemma:lemma7}
  Let $M\in \N$, $q=2^M$, and let $f_1, \ldots, f_\ell \in L^q(\R^n)$ be
  positive functions. Then
  \em
  \begin{IEEEeqnarray}{rCl}
    \bigg\| \sum_{j=1}^\ell L^{(j)} f_j \bigg\|_q
    &\le& C_{q, N} \sum_{k=0}^{M-1} \bigg\| \bigg( \prod_{j=1}^\ell L_j
    \bigg) \bigg( \sum_{j=1}^\ell f_j^{2^k} \bigg)
    \bigg\|_{2^{M-k}}^{2^{-k}}
    + C_{q, N} \bigg( \sum_{j=1}^\ell \| f_j \|_q^q 
    \bigg)^{\nicefrac1q} \label{eq:35} \\
    &\le& C_{q, N} \bigg\| \sum_{j=1}^\ell f_j \bigg\|_q \label{eq:352}.
  \end{IEEEeqnarray}  
  \em
\end{lemma}
\begin{proof}
  The proof of \eqref{eq:352} is easy. We will show \eqref{eq:35} by
  induction on $M$, with the case $M=0$ ($q=1$) being obvious. Fix
  $M>0$ and assume that 
  \begin{equation*}
    \bigg\| \sum_{j=1}^n L^{(j)} f_j \bigg\|_{\nicefrac q2}
    \le C_q \sum_{k=0}^{M-2} \bigg\| \bigg( \prod_{j=1}^n L_j
    \bigg) \bigg( \sum_{j=1}^n f_j^{2^k} \bigg)
    \bigg\|_{2^{M-1-k}}^{2^{-k}}
    + C_q \bigg( \sum_{j=1}^n \| f_j \|_{\nicefrac q2}^{\nicefrac q2}
    \bigg)^{\nicefrac 2q}.
  \end{equation*}
  Then we have
  \begin{equation*}
    \bigg\| \sum_{j=1}^n L^{(j)} f_j \bigg\|_q^q
    \le q! \sum_{1\le j_1 \le \ldots \le j_q \le \ell}\, \int_{\R^n} \prod_{k=1}^q
    \big( L^{(j_k)} f_{j_k}(x) \big) \dd x
  \end{equation*}
  If we split
  \begin{IEEEeqnarray*}{rCl}
    \sum_{1\le j_1 \le \ldots \le j_q \le \ell} \,\int_{\R^n} \prod_{k=1}^q
    \big( L^{(j_k)} f_{j_k}(x) \big) \dd x
    &=& \sum_{\substack{j_1 \le \ldots \le j_q \\ j_1 = j_2}}
    \,\int_{\R^n} \prod_{k=1}^q \big( L^{(j_k)} f_{j_k}(x)
    \big) \dd x \\
    && +\> \sum_{\substack{j_1 \le \ldots \le j_q \\ j_1 < j_2}}
    \,\int_{\R^n} \prod_{k=1}^q \big( L^{(j_k)} f_{j_k}(x)
    \big) \dd x, 
  \end{IEEEeqnarray*}
  we can estimate the first sum as follows, using Hölder's
  inequality with $\frac q2$ and $\frac q{q-2}$
  \begin{IEEEeqnarray*}{rCl}
    \sum_{\substack{j_1 \le \ldots \le j_q \\ j_1 = j_2}}
    \,\int_{\R^n} \prod_{k=1}^q \big( L^{(j_k)} f_{j_k}(x)
    \big) \dd x
    &\le& \int_{\R^n} \bigg( \sum_{j=1}^n \big( L^{(j)}f_j(x) \big)^2
    \bigg) \cdot \bigg( \sum_{j=1}^n L^{(j)}f_j(x) \bigg)^{q-2} \dd x
    \\
    &\le& \bigg\| \sum_{j=1}^n \big( L^{(j)}f_j \big)^2
    \bigg\|_{\nicefrac q2} \bigg\| \sum_{j=1}^n L^{(j)}f_j
    \bigg\|_q^{q-2}. \yesnumber \label{eq:3661}
  \end{IEEEeqnarray*}
  In the case $j_1 < j_2$, we estimate
  \[
  \int_{\R^n} \prod_{k=1}^q \big( L^{(j_k)} f_{j_k}(x) \big) \dd x
  \]
  directly. Without loss of generality, assume $j_1=1$ and put 
  \begin{equation*}
    g_j = \bigg( \prod_{\substack{1\le k\le \ell \\ k \notin \{1,
        j\}}} L_k \bigg) f_j
  \end{equation*}
  for $j \in \ito \ell$. Denote $x = (x^{(1)}, x')$ with $x^{(1)} \in
  \R^N$ and $x' \in \R^{n-N}$, and let $\Phi \colon \R^N \to \R$,
  $\Phi(y) = \prod_{k=1}^N \varphi(y_k)$. Then 
  \begin{equation*}
    \Phi_{t_0}(y) = \prod_{k=1}^N \varphi_{t_0}(y_k),
  \end{equation*}
  and we get
  \begin{IEEEeqnarray*}{rCl}
    \IEEEeqnarraymulticol{3}{l}{
      \int_{\R^n} \prod_{k=1}^q \big( L^{(j_k)} f_{j_k}(x)
      \big) \dd x } \\\quad 
    &=& \int_{\R^{n-N}} \int_{\R^N} g_1(x^{(1)}, x')
    \cdot \prod_{k=2}^q L_1g_{j_k}(x^{(1)}, x') \dd x^{(1)} \dd x' \\
    &=& \int_{\R^{n-N}} \int_{(\R^N)^{q-1}} \int_{\R^N}
    g_1(x^{(1)}, x') \cdot \prod_{k=2}^q \big(
    g_{j_k}(x^{(1)}-y^{(k)}, x') \Phi_{t_0}(y^{(k)}) \big) \dd x^{(1)}
    \dd(y^{(2)}, \ldots, y^{(q)}) \dd x'.
  \end{IEEEeqnarray*}
  Now fix $x'$. Using that $\varphi_{t_0}(\tau) \ge \varphi_{t_0}(t_0)
  = \frac c{2t_0}$, averaging over $|\tau| \le t_0$ implies 
  \begin{IEEEeqnarray*}{rCl}
    \IEEEeqnarraymulticol{3}{l}{
      \int_{(\R^N)^{q-1}} \int_{\R^N}
    g_1(x^{(1)}, x') \cdot \prod_{k=2}^q \big(
    g_{j_k}(x^{(1)}-y^{(k)}, x') \Phi_{t_0}(y^{(k)}) \big) \dd x^{(1)}
    \dd(y^{(2)}, \ldots, y^{(q)}) }\\\quad
    &\le& c^{-N} \int_{\mathopen[-t_0, t_0 \mathclose]^N}
    \Phi_{t_0}(\tau) \int_{(\R^N)^{q-1}} \int_{\R^N} g_1(x^{(1)}, x')
    \\ 
    && \cdot\> \prod_{k=2}^q \big( g_{j_k}(x^{(1)}-y^{(k)}, x')
    \Phi_{t_0}(y^{(k)}) \big) \dd x^{(1)} \dd(y^{(2)}, \ldots,
    y^{(q)}) \dd \tau\\
    &=& c^{-N} \int_{\mathopen[-t_0, t_0 \mathclose]^N}
    \Phi_{t_0}(\tau) \int_{(\R^N)^{q-1}} \int_{\R^N} g_1(x^{(1)}-\tau,
    x') \\
    && \cdot\> \prod_{k=2}^q \big( g_{j_k}(x^{(1)}-y^{(k)}, x')
    \Phi_{t_0}(y^{(k)} + \tau) \big) \dd x^{(1)} \dd(y^{(2)}, \ldots,
    y^{(q)}) \dd \tau\\
    &\le& C^{N(q-1)} c^{-N} \int_{\mathopen[-t_0, t_0 \mathclose]^N}
    \Phi_{t_0}(\tau) \int_{(\R^N)^{q-1}} \int_{\R^N} g_1(x^{(1)}-\tau,
    x') \\
    && \cdot\> \prod_{k=2}^q \big( g_{j_k}(x^{(1)}-y^{(k)}, x')
    \Phi_{t_0}(y^{(k)}) \big) \dd x^{(1)} \dd(y^{(2)}, \ldots,
    y^{(q)}) \dd \tau\\
    &\le& C_{q, N} \int_{\R^N} \prod_{k=1}^q \big(L_1g_{j_k}(x^{(1)},
    x') \big) \dd x^{(1)}.
  \end{IEEEeqnarray*}
  Altogether, we have
  \begin{IEEEeqnarray*}{rCl}
    \int_{\R^n} \prod_{k=1}^q \big( L^{(j_k)} f_{j_k}(x)
    \big) \dd x
    &\le& C_{q, N} \int_{\R^n} \prod_{k=1}^q \big(L_1g_{j_k}(x)
    \big) \dd x \\
    &=& C_{q, N} \int_{\R^n} \bigg(\prod_{j=1}^n L_j \bigg) f_1(x)
    \cdot \prod_{k=2}^q L^{(j_k)} f_{j_k}(x) \dd
    x.
  \end{IEEEeqnarray*}
  The same argument holds if $j_1 \neq 1$, and hence by Hölder's
  inequality with $q$ and $\frac{q-1}q$,
  \begin{IEEEeqnarray*}{rCl}
    \sum_{\substack{j_1 \le \ldots \le j_q \\ j_1 < j_2}}
    \int \prod_{k=1}^q \big( L^{(j_k)} f_{j_k}(x)
    \big) \dd x
    &\le& C_{q, N} \sum_{j=1}^n \sum_{1\le j_2, \ldots, j_q \le \ell}
    \,\int_{\R^n} \bigg(\prod_{j=1}^n L_j \bigg) f_j(x) \cdot
    \prod_{k=2}^q L^{(j_k)} f_{j_k}(x) \dd x \\
    &\le& C_{q, N} \bigg\| \bigg(\prod_{j=1}^n L_j \bigg) \bigg(
    \sum_{j=1}^n f_j \bigg) \bigg\|_q \bigg\| \sum_{j=1}^n L^{(j)} f_j
    \bigg\|_q^{q-1}. \yesnumber \label{eq:3662}
  \end{IEEEeqnarray*}
  With \eqref{eq:3661} and \eqref{eq:3662}, we get
  \begin{IEEEeqnarray*}{rCl}
    \bigg\| \sum_{j=1}^n L^{(j)} f_j \bigg\|_q^q
    &\le& C_q \bigg\| \sum_{j=1}^n \big( L^{(j)}f_j \big)^2
    \bigg\|_{\nicefrac q2} \bigg\| \sum_{j=1}^n L^{(j)}f_j
    \bigg\|_q^{q-2}
    \\
    && +\> C_q \bigg\| \bigg(\prod_{j=1}^n L_j \bigg) \bigg(
    \sum_{j=1}^n f_j \bigg) \bigg\|_q \bigg\| \sum_{j=1}^n L^{(j)} f_j
    \bigg\|_q^{q-1}, 
  \end{IEEEeqnarray*}
  Since 
  \begin{equation*}
    \bigg\| \sum_{j=1}^n \big( L^{(j)}f_j \big)^2 \bigg\|_{\nicefrac
      q2}^{\nicefrac12} \le \bigg\| \sum_{j=1}^n L^{(j)}f_j \bigg\|_q
  \end{equation*}
  and $(L^{(j)}f_j)^2 \le L^{(j)}f_j^2$, this leads to
  \begin{equation*}
    \bigg\| \sum_{j=1}^n L^{(j)} f_j \bigg\|_q
    \le C_q \bigg\| \sum_{j=1}^n L^{(j)}f_j^2
    \bigg\|_{\nicefrac q2}^{\nicefrac12}
    + C_q \bigg\| \bigg(\prod_{j=1}^n L_j \bigg) \bigg( \sum_{j=1}^n
    f_j \bigg) \bigg\|_q,
  \end{equation*}
  and our induction hypothesis concludes the lemma.
\end{proof}
We directly show a version of Lemma 9 from \cite{bou2}, which is a
corollary of Lemma 8.
\begin{lemma}
  \label{lemma:lemma9}
  Let $M \ge 1$, $q = 2^M$, and $f_1, \ldots, f_\ell \in
  L^q(\R^n)$. Then
  \begin{equation}
    \label{eq:39}
    \bigg\| \bigg( \sum_{k=1}^\ell |L^{(k)} G_k f_k |^2
    \bigg)^{\nicefrac12} \bigg\|_q
    \le C_{q, N} R^{12 N \cdot \eps} \bigg\| \bigg( \sum_{k=1}^\ell | f_k
    |^2 \bigg)^{\nicefrac12} \bigg\|_q.
  \end{equation}
\end{lemma}
\begin{proof}
  Since $\mu_j = \del_j(\chi_B \ast H_{\nicefrac1R}) = \del_j(\chi_B)
  \ast H_{\nicefrac1R}$ by taking distributional derivatives and
  convolutions, and $H_{\nicefrac1R}$ is the density function of a
  probability measure, we can use Jensen's inequality to estimate
  \begin{IEEEeqnarray*}{rCl}
    \bigg\| \bigg( \sum_{k=1}^\ell |L^{(k)} G_k f_k |^2
    \bigg)^{\nicefrac12} \bigg\|_q
    &\le& \bigg\| H_{\nicefrac1R} \ast \bigg( \sum_{k=1}^\ell |L^{(k)}
    G_k' f_k|^2 \bigg)^{\nicefrac12} \bigg\|_q
    \\ 
    &\le& \bigg\| \bigg( \sum_{k=1}^\ell L^{(k)}
    |G_k' f_k|^2 \bigg)^{\nicefrac12} \bigg\|_q,
  \end{IEEEeqnarray*}
  where
  \begin{equation*}
    G_k' f_k = \bigg( \sum_{j \in V_k} | (\del_j \chi_B) \ast f_k |^2
    \bigg)^{\nicefrac12}.
  \end{equation*}
  Hence it suffices to show
  \begin{equation}
    \label{eq:399}
    \bigg\| \bigg( \sum_{k=1}^\ell L^{(k)}
    |G_k' f_k|^2 \bigg)^{\nicefrac12} \bigg\|_q
    \le C_q R^{24 N \eps} \bigg\| \bigg( \sum_{k=1}^\ell | f_k |^2
    \bigg)^{\nicefrac12} \bigg\|_q.
  \end{equation}
  Take $\psi \in \mS(\R^n)$. We want to establish $\scp{\del_j\chi_B}
  {\psi}$ for every $j \in \ito n$. First, assume $j = \ell = 1$. Then
  we have
  \begin{IEEEeqnarray*}{rCl}
    -\scp{\del_j\chi_B}{\psi}
    = \int_{B_{r_N}^N} \del_1 \psi(x) \dd x
    &=& \int_{B^{N-1}_{r_N}} \int_{-\sqrt{r_N^2 -
        |x'|^2}}^{\sqrt{r_N^2 - |x'|^2}} \del_1 \psi(x_1, x') \dd x_1
    \dd x' \\
    &=& \int_{B^{N-1}_{r_N}} \psi \big(\sqrt{r_N^2 - |x'|^2}, x'
    \big) - \psi \big({-}{\sqrt{r_N^2 - |x'|^2}}, x' \big) \dd x' \\
    &=& \int_{S^{N-1}_{r_N}} \psi(\omega) \frac{\omega_1}{r_N} \dd
    \sigma(\omega).
  \end{IEEEeqnarray*}
  For general $j$ and $\ell$, choose the unique $k=k(j)$ with $j \in
  V_{k}$. Then
  \begin{equation}
    \label{eq:distderi}
    -\scp{\del_j\chi_B}{\psi}
    = \int_{B^{\hat k}} \int_{S^{N-1}_{r_N}} \psi(x^{(1)}, \ldots,
    \underbrace{\omega}_{k \text{-th}}, \ldots, x^{(\ell)})
    \frac{\omega_{j-(k-1)N}}{r_N} \dd \sigma(\omega) \dd x',
  \end{equation}
  where $B^{\hat k} = \prod\limits_{k' \neq k} B_{k'}$ and $x' =
  (x_1, \ldots, x_{(k-1)N},  x_{kN+1}, \ldots, x_n)$.
  Now let
  \begin{equation*}
    \tau_j f(x) = \int_{S^{N-1}_{r_N}} f(x^{(1)}, \ldots,
    x^{(k)}+\omega, \ldots, x^{(\ell)}) \frac{|\omega_{j-(k-1)N}|}{r_N}
    \dd \sigma(\omega)
  \end{equation*}
  Then
  \begin{equation*}
    \| \tau_j\|_{1 \to 1} = 2 |B^{N-1}_{r_N}|,
  \end{equation*}
  and hence
  \begin{equation*}
    \sum_{j \in V_k} |(\del_j \chi_B) \ast f_k|^2
    \le \sum_{j \in V_k} \big| \chi_{B^{\hat k}} \ast \tau_j|f_k| \big|^2
    \le 2 |B^{N-1}_{r_N}| \sum_{j \in V_k}\tau_j(|f_k|^2 \ast
    \chi_{B^{\hat k}}),
  \end{equation*}
  taking the last convolution only in the variables of $B^{\hat
    k}$. Application of Lemma \ref{lemma:lemma7} gives us
  \begin{IEEEeqnarray*}{rCl}
    \bigg\| \bigg( \sum_{k=1}^\ell L^{(k)}
    |G_k' f_k|^2 \bigg)^{\nicefrac12} \bigg\|_q
    &\le& 2 |B^{N-1}_{r_N}| \bigg\| \sum_{k=1}^\ell L^{(k)}
    \bigg(\sum_{j \in V_k} \tau_j \big( |f_j|^2 \ast \chi_{B^{\hat k}}
    \big) \bigg) \bigg\|_{\nicefrac q2}^{\nicefrac12} \\ 
    &\le& C_{q, N} \sum_{i=1}^{M-1} \bigg\| \bigg( \prod_{k=1}^\ell L_k
    \bigg) \bigg( \sum_{k=1}^\ell \bigg( \sum_{j \in V_k} \tau_j (
    |f_k|^2 \ast \chi_{B^{\hat k}} ) \bigg) ^{2^{i-1}} \bigg)
    \bigg\|_{2^{M-i}}^{2^{-i}} \\
    && +\> C_{q, N} \bigg( \sum_{k=1}^\ell \bigg\| \sum_{j \in V_k}
    \tau_j ( |f_k|^2 \ast \chi_{B^{\hat k}}) \bigg\|_{\nicefrac
      q2}^{\nicefrac q2} \bigg)^{\nicefrac 1q}.
    \yesnumber \label{eq:310} 
  \end{IEEEeqnarray*}
  We can easily estimate
  \begin{equation}
    \label{eq:311}
    \bigg( \sum_{k=1}^\ell \bigg\| \sum_{j \in V_k}
    \tau_j ( |f_k|^2 \ast \chi_{B^{\hat k}}) \bigg\|_{\nicefrac
      q2}^{\nicefrac q2} \bigg)^{\nicefrac 1q}
    \le (2 N |B^{N-1}_{r_N}|)^{\nicefrac12} \cdot \bigg\| \bigg(
    \sum_{k=1}^\ell |f_k(x)|^2 \bigg)^{\nicefrac12} \bigg\|_q.
  \end{equation}
  Note that from the properties of $\varphi$, we can deduce
  \begin{equation*}
    \varphi_{t_0}(x+\rho) = \frac1{t_0} \frac{c}{1 +
      \frac{(x+\rho)^4}{t_0^4}}
    \le \frac1{t_0} \varphi(x+\rho)
    \le \frac {C^{\lceil r_N \rceil}}{t_0} \varphi(x)
    = \frac {C_N}{t_0^5} \frac1{\frac{1+x^4}{t_0^4}}
    \le \frac {C_N}{t_0^4} \varphi_{t_0}(x)
  \end{equation*}
  for $x \in \R$ and $|\rho| \in \mathopen[-r_N, r_N
  \mathclose]$. Hence
  \begin{equation*}
    \Phi_{t_0}(x-y) \le C_N t_0^{-4N} \Phi_{t_0}(x)
  \end{equation*}
  for every $x, y \in \R^N$ with $|y| \le r_N$. Thus for any positive
  $g \in L^q(\R^N)$, we have 
  \begin{IEEEeqnarray*}{rCl}
    \bigg(\Phi_{t_0} \ast \int_{S^{N-1}_{r_N}} (\tau_\omega g) \cdot
    \frac{|\omega_j|}{r_N} \dd \sigma(\omega) \bigg) (x)
    &\le& 2 |B^{N-1}_{r_N}| C_N t_0^{-4N}  \big( \Phi_{t_0} \ast g(x)
    \big) \\
    &=& C_N t_0^{-4N} \int_{\R^N} \int_{B_{r_N}^N} \Phi_{t_0}(y) \dd
      z\, g(x-y) \dd y\\
    &\le& C_Nt_0^{-8N} \int_{\R^N} \int_{B_{r_N}^N} \Phi_{t_0}(y-z) \dd
      z\, g(x-y) \dd y\\ 
    &=& C_NR^{24N \cdot \eps} \cdot (\chi_{B_{r_N}^N} \ast
    \Phi_{t_0}) \ast g(x). 
  \end{IEEEeqnarray*}
  Taking $1 \le i \le N-1$, this implies
  \begin{IEEEeqnarray*}{rCl}
    \IEEEeqnarraymulticol{3}{l}{
      \bigg\| \bigg( \prod_{k=1}^\ell L_k \bigg) \bigg(
      \sum_{k=1}^\ell \bigg( \sum_{j \in V_k} \tau_j ( |f_k|^2 \ast
      \chi_{B^{\hat k}} ) \bigg) ^{2^{i-1} } \bigg)
      \bigg\|_{2^{M-i}}^{2^{-i}} } \\ \quad 
    &\le& \bigg\| \bigg( \sum_{k=1}^\ell C_N N R^{24N
     \cdot \eps} \cdot |f_k|^2 \ast \chi_B ) \bigg) ^{2^{i-1}}
   \bigg\|_{2^{M-i}}^{2^{-i}} \\  
   &\le& C_N R^{12N \cdot \eps} \bigg\|\bigg( \sum_{k=1}^\ell
   |f_k|^2 \bigg)^{\nicefrac12} \bigg\|_q. \yesnumber \label{eq:313}
  \end{IEEEeqnarray*}
  The estimates \eqref{eq:310}, \eqref{eq:311}, and \eqref{eq:313}
  conclude the proof.
\end{proof}
Now, we can estimate \eqref{eq:b00} and \eqref{eq:b11} by arguing as
in \cite[Section~4]{bou2}. Since
\begin{equation*}
  A_0 = \prod_{k=1}^\ell S_k = \prod_{j=1}^n T_j,
\end{equation*}
we can establish
\begin{equation*}
  b_0 \le C_p(R^{3\eps} + B_p R^{-\frac{2\eps}p})
\end{equation*}
with $B_p$ as in \eqref{eq:ap}.
For \eqref{eq:b11}, we can use Lemmas
\ref{lemma:lemma7} and \ref{lemma:lemma9} to obtain
\begin{equation*}
  b_1 \le C_{p, N}(R^{12N \cdot \eps} + B_p R^{-\frac{2\eps}p}).
\end{equation*}
This leads to
\begin{equation*}
  B_p \le C_{p, \eps}(1 + b_0 + b_1) < C_{p, \eps, N}(R^{12N \cdot
    \eps} + B_p R^{-\frac{2\eps}p}),
\end{equation*}
giving us
\begin{equation}
  \label{eq:bpfinal}
  B_p \le C_{p, \eps, N} R^{12N \cdot \eps}
\end{equation}
and thus proving Lemma \ref{lemma:lemma3} and Proposition \ref{prop:bouN}.

\section{Stein's approach revisited}
\label{part3}
Let $B = B_1 \times \cdots \times B_\ell$, $n = n_1 + \ldots + n_\ell$
be as in Theorem \ref{theo:one}. Let $N \ce \min\limits_{1 \le k \le
  \ell} n_k$. Fix $1 < p \le \infty$. We show that if $N >
\frac{p}{p-1}$, i.e. $p> \frac{N}{N-1}$, we can deduce
\eqref{eq:theone} in Theorem \ref{theo:one} from Theorem
\ref{theo:two}, and that \eqref{eq:thetwo} from Theorem
\ref{theo:two} holds if we show the following theorem.
\setcounter{theopr}{1}
\begin{theopr}
  \label{theo:twopr}
  Let $S \ce \big(S^{N-1}_R\big)^\ell$, with $R$ so that $|S| =
  1$. Then
  \begin{equation}
  \label{eq:theneededone}
    \norm{M_S f}_p \le C_{p, N} \norm{f}_p.
  \end{equation}
\end{theopr}
We can freely change the radii of each sphere because for every
$r, s > 0$ and $n > 1$, we have
\begin{equation*}
  \int_{S^{n-1}_r} f(s \omega) \dd \sigma_{S^{n-1}_r}(\omega)
  = \frac1{s^{n-1}} \int_{S^{n-1}_{sr}} f(\omega) \dd
  \sigma_{S^{n-1}_{sr}}(\omega).
\end{equation*}
Thus, if $S' = S_1' \times \cdots \times S_\ell'$ is a product of
spheres with $\dim S_k' = \dim S_k$ for each $k$, we have $\| M_{S'}f
\|_p \le C \| f \|_p$ for all $f \in \mS$ if $\| M_Sf \|_p \le C
\| f \|_p$ for all $f \in \mS$.\\
We can not get a pointwise estimate $M_Bf(x) \le M_Sf(x)$ as in the
case $\ell = 1$, but we can indeed get an $L^p$-estimate by a similar
argument. Assume that each $B_k$ and each $S_k$ has radius $1$ (thus
$S_k = S^{n_k-1}$), and let $\sigma_k$ be the respective surface
measure for each $S_k$.
\begin{lemma}
  \label{lemma:mbms}
  We have
  \begin{equation}
    \label{eq:mbms}
    \| M_B f \|_p \le \| M_S f \|_p
  \end{equation}
  for each $f \in \mS(\R^n)$.
\end{lemma}
\begin{proof}
  Using polar coordinates and the fact that
  $|S_k| = n_k|B_k|$, we can estimate
  \begin{IEEEeqnarray*}{rClr}
    M_Bf(x)
    &=& \frac1{|B|} \sup_{t>0} \int_{[0, 1]^\ell} \prod_{k=1}^\ell
    s_k^{n_k-1} \int_{S_1} \cdots \int_{S_\ell} |f(x + t(s_1\omega_1,
    \ldots, s_\ell\omega_\ell))| \dd \sigma_\ell(\omega_\ell) \ldots &
    \dd \sigma_1(\omega_1) \dd s \\
    &\le& \int_{[0, 1]^\ell} \prod_{k=1}^\ell s_k^{n_k-1} \frac1{|B|}
    \prod_{k=1}^\ell s_k^{-n_k+1} \sup_{t>0} \int_{S^{n_1-1}_{s_1}}
    \cdots \int_{S^{n_\ell-1}_{s_\ell}} |f(x + t(\omega_1, \ldots,
    \omega_\ell))|\\
    \IEEEeqnarraymulticol{3}{r}{ \dd \sigma_\ell(\omega_\ell)
      \ldots }&\dd \sigma_1(\omega_1) \dd s \\
    &=& \int_{[0, 1]^\ell} \prod_{k=1}^\ell n_k s_k^{n_k-1}
    M_{S^{n_1-1}_{s_1} \times \cdots \times S^{n_\ell-1}_{s_\ell}} f(x)
    \dd s.
  \end{IEEEeqnarray*}
  A simple application of
  Minkowski's integral inequality yields
  \begin{equation*}
    \| M_Bf \|_p
    \le \int_{[0, 1]^\ell} \prod_{k=1}^\ell n_k s_k^{n_k-1} \cdot
    \|M_Sf\|_p \dd s 
    = \|M_Sf\|_p,
  \end{equation*}
  which is \eqref{eq:mbms}.
\end{proof}
We now generalize Stein's method of rotations from \cite{ste} for our
situation with the following Lemma 
\begin{lemma}
  \label{lemma:stein}
  Let $S = S^{n_1 - 1} \times \cdots \times S^{n_\ell - 1}$, $k \in
  \ito \ell$, and set
  \begin{equation*}
    S^+ = \prod_{j=1}^{k-1} S^{n_j-1} \times S^{n_k} \times
    \prod_{j=k+1}^\ell S^{n_j-1}.
  \end{equation*}
  Let $1< p \le \infty$ and assume that there is a constant $C>0$ such
  that $\| M_S f \|_p \le C \| f \|_p$ for every $f \in
  \mS(\R^n)$. Then also
  \begin{equation}
    \label{eq:lemmasec3final}
    \| M_{S^+} f \|_p \le C \| f \|_p
  \end{equation}
  for every $f \in \mS(\R^n)$.
\end{lemma}
\begin{proof}
  We can assume $k=1$. For any $u \in S^{n_1}$, denote by
  \begin{equation*}
    S^{n_1-1}_u \ce \{ x \in S^{n_1} : x \perp u \}
  \end{equation*}
  the rotated $(n_1-1)$-dimensional spheres in $\R^{n_1+1}$, and let
  $\sigma^u$ be the surface measure of $S^{n_1-1}_u$ so that
  $\sigma^u(S^{n_1-1}_u) = |S^{n_1-1}|$. Furthermore, let $\sigma_1^+$
  be the surface measure of $S^{n_1}$. Define a new measure $\mu$ on
  $S^{n_1}$ by putting for every Lebesgue-measurable set $A \subset
  \R^{n_1+1}$
  \begin{equation*}
    \mu(A) \ce \int_{S^{n_1}} \int_{S^{n_1-1}_u} \chi_A(\omega) \dd
    \sigma^u(\omega) \dd \sigma_1^+(u).
  \end{equation*}
  By \cite{ste}, we have $\mu = |S^{n_1-1}| \cdot \sigma_1^+$, and
  \begin{equation*}
    \| M_{S^{n_1-1}_u} f \|_p = \| M_{S^{n_1-1}} f \|_p \le C \|f\|_p,
  \end{equation*}
  where $S_u = S^{n_1-1}_u \times S_2 \times \cdots \times
  S_\ell$. Hence we can calculate 
  \begin{IEEEeqnarray*}{rClr}
    M_{S^+}f(x)
    &=& \frac1{|S^+|} \sup_{t>0} \frac1{|S^{n_1-1}|} \int_{S^{n_1}}
    \int_{S^{n_1-1}_u} \int_{S_2} \cdots \int_{S_\ell} | f (x +
    t(\omega_1, \ldots \omega_\ell)) |& \\
    \IEEEeqnarraymulticol{3}{r}{ \dd \sigma_\ell(\omega_\ell)
      \ldots \dd \sigma_2(\omega_2) }& \dd
    \sigma_1^u(\omega_1) \dd \sigma_1^+(u)\\
    &\le& \frac1{|S^{n_1}|} \int_{S^{n_1}} \frac1{|S|} \sup_{t>0}
    \int_{S^{n_1-1}_u} \int_{S_2} \cdots \int_{S_\ell} | f (x +
    t(\omega_1, \ldots \omega_\ell)) | &\\
    \IEEEeqnarraymulticol{3}{r}{\dd \sigma_\ell(\omega_\ell)
      \ldots \dd \sigma_2(\omega_2) }&\dd
    \sigma_1^u(\omega_1) \dd \sigma_1^+(u) \\
    &=& \frac1{|S^{n_1}|} \int_{S^{n_1}} M_{S_u}f(x) \dd
    \sigma_1^+(u),
  \end{IEEEeqnarray*}
  By Minkowski's integral inequality, we then get
  \begin{equation*}
    \| M_{S^+}f \|_p
    \le \frac1{|S^{n_1}|} \int_{S^{n_1}} \| M_{S_u}f \|_p \dd
    \sigma_1^+(u)
    \le C \| f \|_p.
  \end{equation*}
  This concludes the lemma.
\end{proof}
Now assume that we've already shown Theorem \ref{theo:twopr} and take
$B$ as in Theorem \ref{theo:one}. Fix $1 < p < \infty$ and let $N_0 =
\lceil \frac{p}{p-1} \rceil$. By change of coordinates, we can assume
$B = B' \times B''$, where
\begin{equation*}
  B' = \prod_{n_k \ge N_0} B_k \quad\text{and}\quad
  B'' = \prod_{n_k < N_0} B_k.
\end{equation*}
Then
\begin{equation*}
  \| M_{B} \|_{p \to p} \le \| M_{B'} \|_{p \to p} \cdot \| M_{B''}
  \|_{p \to p}. 
\end{equation*}
Furthermore, assume that for $n_k \ge N_0$, each $B_k$ has radius $1$,
while for $n_k < N_0$, each $B_k$ has volume $1$. Since $N_0$ only
depends on $p$, we get
\begin{equation*}
  \| M_{B''} \|_{p \to p} \le C_p
\end{equation*}
from Proposition \ref{prop:bouN}. Let $S' = \prod_{n_k \ge N_0}
S^{n_k-1}$. Then by Lemma \ref{lemma:mbms} and successive application of
Lemma \ref{lemma:stein}, we obtain
\begin{equation*}
  \| M_{B'} \|_{p \to p} \le \| M_{S'} \|_{p \to p} \le \|
  M_{(S^{N_0-1})^{\ell'}} \|_{p \to p},
\end{equation*}
where $\ell' = |\{k \in \ito \ell : n_k \ge N_0\}|$. But since we can
freely vary the radius of each sphere, Theorem
\ref{theo:twopr} implies 
\begin{equation*}
  \| M_{(S^{N_0-1})^{\ell'}} \|_{p \to p} \le C_{N_0} = C_p.
\end{equation*}
This concludes Theorem \ref{theo:one}.
\section{Higher Fourier derivatives of spherical measures}
\label{part4}
We are left with proving Theorem \ref{theo:twopr}. Let $N > 2$, $p >
\frac N{N-1}$, and $S=(S^{N-1}_R)^\ell$ with $R$ so that $|S|=1$,
i.e. $R^{N-1} = \frac{\Gamma(\nicefrac N2)}{2\pi^{\nicefrac N2}}$.\\
First, we use the approach from \cite{car} to show that we only have
to bound
\begin{equation}
  \label{eq:tonystrick}
  \big\| \sup\limits_{1 \le t \le 2} \int_{S} |f(x + t \omega)| \dd
  \sigma_S(\omega) \big \|_p \le C_{p, N} \| f \|_p
\end{equation}
for $p < 2$. For our setting, we use a different proof, which can be
found in Lemma 6.15 and the argument in subsection 6.5.1 from
\cite{dgm}. This result makes use of Lemma 3 from \cite{bou1}. Since
the proofs only rely on the properties of the respective Fourier
transforms, a short inspection of them shows that these lemmas still
hold when we take finite (signed) Borel measures on $\R^n$ instead of
$L^1$-kernels, in the following sense.
\begin{lemma}[Lemma 3 from \cite{bou1}]
  \label{lemma:l3b1}
  Let $\nu$ be a finite Borel measure on $\R^n$ so that $\hat \nu$ is
  differentiable, and put
  \begin{equation*}
    \alpha_j \ce \sup_{2^j \le |\xi| \le 2^{j+2}}|\hat \nu(\xi)|,\quad
    \beta_j \ce \sup_{2^j \le |\xi| \le 2^{j+2}} |\scp{\nabla\hat \nu(\xi)}\xi|
  \end{equation*}
  for every $j\in\Z$. Then for every $f\in L^2$, we have
  \begin{equation*}
    \big\|\sup_{t>0} |f\ast \nu_t| \big\|_2 \le C\Gamma(\nu)\|f\|_2
  \end{equation*}
  with
  \begin{equation}
    \label{eq:gamma}
    \Gamma(\nu) \ce \sum_{j\in \Z} \alpha_j^{\nicefrac12}
    (\alpha_j^{\nicefrac12}+\beta_j^{\nicefrac12}).
  \end{equation}
\end{lemma}
\begin{lemma}[Lemma 6.15 from \cite{dgm}]
  \label{lemma:dgm615}
  Let $\nu$ be a finite Borel measure on $\R^n$ and $K \in L^1(\R^n)$
  such that $\hat \nu$ and $\hat K$ are both differentiable. Assume
  there is a constant $C$ so that for every $\theta \in S^{n-1}$ and
  every $u \in \R^*$
  \em
  \begin{IEEEeqnarray}{rCl}
    | \hat \mu(u \theta) | &\le& C \cdot \min\{|u|,
    |u|^{-1}\}, \label{eq:erstedgm} \\
    | \scp{\theta}{\nabla \hat \mu(u \theta)} | &\le& C \cdot
    \min\{1, |u|^{-1}\} \label{eq:zweitedgm}
  \end{IEEEeqnarray}
  \em
  with $\mu = \nu$, $K$.
  Then we have
  \begin{equation}
    \label{eq:carbend}
    \Gamma(\nu \ast K_{2^k}) \le C' 2^{-\nicefrac{|k|}2}
  \end{equation}
  for every $k \in \Z$, with $C'$ only depending on $C$ and
  $\Gamma(\nu \ast K_{2^k})$ as in \eqref{eq:gamma}.
\end{lemma}
Let $P$ be the Poisson kernel, i.e. $\hat P(\xi) = e^{-|\xi|}$. We
will show later that the Borel measure $\sigma_S - P \dd x$
satisfies \eqref{eq:erstedgm} and \eqref{eq:zweitedgm}. By Stein's
maximal theorem for semigroups, $\| \sup_{t>0} P_t \ast f \|_q <
C\|f\|_q$ for $1<q<\infty $, and one can take $\hat K(\xi) =
e^{-|\xi|}-e^{-2|\xi|}$ in Lemma \ref{lemma:dgm615} to proceed as in
\cite{dgm}, getting the strong $L^2$-boundedness property from
Carbery's proof as required, and being left with showing
\eqref{eq:tonystrick}. Here, we use part (ii) of the proposition in
\cite{car}, which also holds for any finite Borel measure with bounded
Fourier transform. For this, we need to consider fractional
derivatives. For a finite Borel measure $\nu$ on $\R^n$ and $z \in
\C$, denote the fractional derivative of $\hat \nu$ of order $z$ by
\begin{equation}
  \label{eq:fracderi}
  (\scp \xi \nabla)^z \hat \nu(\xi) = \bigg(\frac{\dd}{\dd r}\bigg)^z
  \hat \nu(r\xi)\big|_{r=1} = \int (2\pi i \scp x
  \xi)^z e^{2\pi i \scp x \xi} \dd \nu(x)
\end{equation}
whenever the right hand side is well-defined. Let
\begin{equation*}
  m(\xi) = \widehat{ \sigma_S}(\xi).
\end{equation*}
According to the proposition in
\cite{car}, we need to show that there is $1 > \alpha >
\nicefrac1p$ so that the fractional derivative $\scp \xi \nabla
^\alpha m$ has bounded $L^p$-multiplier norm independent of $\ell
$. Our basic idea will be to estimate the fractional derivatives of
$m$ of order $z$ with $\Re z = 0$ as $L^1$-multipliers, and with $\Re
z = \frac{N-1}2$ as $L^2$-multipliers, followed by applying Stein's
interpolation theorem. It turns out that $\frac{N-1}2$ is the best
possible upper bound on $\Re z$ for that estimate, and that we can
also establish \eqref{eq:erstedgm} and \eqref{eq:zweitedgm} while
bounding these fractional derivatives. However, we encounter some
technical difficulties like in \cite{mue}. To deal with these, we
introduce the Riesz fractional derivative of a function $f \colon
\mathopen] 0, 2 \mathclose] \to \C$, defined as 
\begin{equation*}
  I^{-z}f(t) = \frac{-1}{\Gamma(-z)} \int_t^2 (u-t)^{-z-1}f(u) \dd u,
\end{equation*}
for $\Re z < 0$ and $0 < t \le 2$. This operator can be extended
analytically to the complex plane. For any $k \in \N_{>0}$, assume
that $f$ as above is $k$ times differentiable and let $\Re z <
k$. Then
\begin{equation}
  \label{eq:rieszfd}
  I^{-z}f(t) = E_{k,f}(z, t) + (-1)^k \frac1{\Gamma(k-z)} \int_t^2
  (u-t)^{-z+k-1}f^{(k)}(u) \dd u,
\end{equation}
where
\begin{equation}
  \label{eq:ekzt}
  E_{k,f}(z,t) = \sum_{j=0}^{k-1} (-1)^j \frac{(2-t)^{-z+j}
    f^{(j)}(2)}{\Gamma(j+1-z)}.
\end{equation}
It follows that
\begin{equation*}
  I^{-k}f(t) = (-1)^kf^{(k)}(t)
\end{equation*}
if $f$ is $k$ times differentiable.
Now consider the holomorphic family of multipliers $(m_z)_{z \in \C}$
defined by
\begin{equation}
  m_z(\xi) \ce I^{-z}m(t \xi) \big\vert_{t=1}.
\end{equation}
From Müller's work (see also Lemma 7.3 in \cite{dgm}), it follows that
for every $0 < \alpha < 1$,
\[
  m_\alpha(\xi) - (\scp \xi \nabla)^\alpha m(\xi)
\]
is an $L^q$-multiplier for $1 \le q \le \infty $, bounded by
$\frac{1}{\Gamma(1-\alpha)}$. Thus we only need to bound the
$m_z$, and from \eqref{eq:rieszfd} and \eqref{eq:ekzt}, it follows
that we need to bound the usual derivatives $\frac{\dd^k}{\dd r^k}
m(r\xi)$ for $1 \le r \le 2$.\\\\
Let us give a quick idea of our approach to bound these derivatives.
We have
\begin{equation*}
  m(\xi) = \prod_{j=1}^\ell \tilde m(\zeta_j),
\end{equation*}
writing again $\xi = (\zeta_1, \ldots, \zeta_\ell)$ with each $\zeta_j
\in \R^N$, and
\begin{equation*}
  \tilde m(\zeta) \ce \widehat{ \sigma_{S^{N-1}_R}}(\zeta).
\end{equation*}
By the general Leibniz formula, we have
\begin{equation}
  \label{eq:leibnizder}
  \bigg( \frac{\dd}{\dd r} \bigg)^k m(r\xi)
  = \sum_{\substack{\alpha \in \N^\ell \\ |\alpha| = k}} {k \choose
    \alpha} \prod_{j=1}^\ell \bigg( \frac{\dd}{\dd r}
  \bigg)^{\alpha_j} \tilde m(r \zeta_j).
\end{equation}
It turns out that the well known bound $\tilde m(\zeta) =
\mO(|\zeta|^{-\frac{N-1}2})$ as $|\zeta| \to \infty$ extends to
$\frac{\dd^\alpha}{\dd r^\alpha} \tilde m(r \zeta) =
\mO(|\zeta|^{\alpha-\frac{N-1}2})$, $|\zeta| \to \infty$. Hence we can
get a good estimate on the factors where $|\zeta_j|$ is sufficiently
large. One can also use the integral representation of Bessel
functions to see that for $k > 0$, $\frac{\dd^k}{\dd r^k}
\tilde m(r \zeta)$ is close to $0$ if $\zeta$ is close to $0$.\\
However, we also need get some estimate for the derivatives of $\tilde
m$ if $\zeta$ is neither sufficiently close to $0$, nor sufficiently
large. All the necessary estimates are established in the following
lemma.
\begin{lemma}
  \label{lemma:derisofsphere}
  For each $\alpha \in \N$ and
  $\zeta \in \R^N$, the following estimates hold.
  \begin{enumerate}[(i)]
  \item There are constants $C_{\alpha, N}, \tilde C_{\alpha, N}$ such
    that for every $|\zeta| \ge C_{\alpha, N}$ and every $r \in
    \mathopen[ 1, 2 \mathclose]$
    \begin{equation*}
      \bigg|\bigg(\frac{\dd}{\dd r}\bigg)^\alpha \tilde
      m(r\zeta)\bigg|
      \le \tilde C_{\alpha, N} ( 2\pi R
      |\zeta|)^{-\frac{N-1}2+\alpha}.
    \end{equation*}
  \item If $0 \le 2\pi R |\zeta| \le C$ for some constant $C > 1$,
    then there is $c> 0$, depending only on $C$, such that for
    every $r \in \mathopen[ 1, 2 \mathclose]$, we have
    \begin{equation*}
      |\tilde m(r\zeta)| \le e^{-c (2\pi R |\zeta|)^2}.
    \end{equation*}
  \item If $\alpha > 0$ and $1 \le 2\pi R |\zeta| \le C$ for some
    constant $C > 1$, then for every $r \in \mathopen[ 1, 2
    \mathclose]$, we have 
    \begin{equation*}
      \bigg|\bigg(\frac{\dd}{\dd r}\bigg)^\alpha \tilde
      m(r\zeta)\bigg|
      \le C' (2\pi R |\zeta|)^{2\alpha}
      \le (2\pi R |\zeta|)^{2\alpha} \cdot e^{-c'(2\pi R |\zeta|)^2} ,
    \end{equation*}
    with $0< C' < 1$, $C'$ only depending on $C$, and $c'$ chosen so
    that $C' \le e^{-c' C^2}$.
  \item If $\alpha > 0$ and $2\pi R |\zeta| < 1$, then there is
    $C' <1$ such that for $1 \le r \le 2$,
    \begin{equation*}
      \bigg|\bigg(\frac{\dd}{\dd r}\bigg)^\alpha \tilde
      m(r\zeta)\bigg|
      \le C' (2\pi R |\zeta|)^2.
    \end{equation*}
  \end{enumerate}
\end{lemma}
\begin{proof}
  (i) \enskip We will show that
  \begin{equation}
    \label{eq:derisofsphere}
    \bigg(\frac{\dd}{\dd r}\bigg)^\alpha \tilde m(r\zeta)
    = 2 \pi^{\alpha+1} r^{-\frac{N-2}2} R^{\frac N2 + \alpha}
    |\zeta|^{-\frac {N-2}2+\alpha} \sum_{k=0}^\alpha {\alpha \choose
      k} |\pi r R \zeta|^{-k} B_{\alpha, k}(2\pi r R |\zeta|),
  \end{equation}
  with
  \begin{equation}
    \label{eq:balphak}
    B_{\alpha, k}(t) =
    \frac{\Gamma(\frac{N-2}2+k)}{\Gamma(\frac{N-2}2)}
    \sum_{j=0}^{\alpha-k} (-1)^{j + k} {{\alpha -k} \choose j}
    J_{\frac{N-2}2 - \alpha + k + 2j}(t).
  \end{equation}
  Since
  \begin{equation*}
    \tilde m(\zeta) = 2\pi R^{\frac N2} |\zeta|^{-\frac{N-2}2}
    J_{\frac{N-2}2}(2\pi R |\zeta|),
  \end{equation*}
  we obtain
  \begin{equation}
    \label{eq:bloed}
    \bigg(\frac{\dd}{\dd r}\bigg)^\alpha \tilde m(r\zeta)
    = (2\pi)^{\frac N2} R^{N-1+\alpha} |\zeta|^\alpha
    \bigg(\frac{\dd}{\dd t}\bigg)^\alpha \big[ (2\pi t)^{-\frac{N-2}2}
      J_{\frac{N-2}2}(2\pi t) \big]\bigg|_{t=rR|\zeta|}.
  \end{equation}
  It is well-known that for every $\nu \in \R$, we have
  \begin{equation*}
    \frac{\dd}{\dd t} J_\nu(t) = \frac12(J_{\nu-1}(t) - J_{\nu+1}(t)),
  \end{equation*}
  which extends to
  \begin{equation*}
    \bigg(\frac{\dd}{\dd t}\bigg)^\alpha J_\nu(t) = \frac1{2^\alpha}
    \sum_{j=0}^\alpha (-1)^j {\alpha \choose j} J_{\nu - \alpha +2j}(t).
  \end{equation*}
  By the Leibniz formula, we thus get
  \begin{IEEEeqnarray*}{rCl}
    \bigg(\frac{\dd}{\dd t}\bigg)^\alpha \big[t^{-\nu}J_\nu(t)\big]
    &=& \sum_{k=0}^\alpha (-1)^k {\alpha \choose k}
    \frac{\Gamma(\nu+k)}{\Gamma(\nu)} t^{-\nu-k} J_\nu^{(\alpha-k)}(t)
    \\
    &=& \sum_{k=0}^\alpha \sum_{j=0}^{\alpha-k} 2^{-\alpha+k}
    (-1)^{j+k} {\alpha \choose k} \frac{\Gamma(\nu+k)}{\Gamma(\nu)}
    t^{-\nu-k} {{\alpha - k} \choose j} J_{\nu-\alpha+k+2j}(t),
  \end{IEEEeqnarray*}
  which we can insert into \eqref{eq:bloed}. For $\nu=\frac{N-2}2$,
  combining this equation with \eqref{eq:balphak} yields
  \begin{equation*}
    \bigg(\frac{\dd}{\dd t}\bigg)^\alpha \big[ (2\pi t)^{-\frac{N-2}2}
    J_{\frac{N-2}2}(2\pi t) \big]
    = (2\pi)^\alpha \sum_{k=0}^\alpha 2^{-\alpha+k} {\alpha \choose k}
    (2\pi t)^{-\frac{N-2}2 +k} B_{\alpha, k}(2\pi t).
  \end{equation*}
  Thus \eqref{eq:bloed} gives us
  \begin{IEEEeqnarray*}{rCl}
    \bigg(\frac{\dd}{\dd r}\bigg)^\alpha \tilde m(r\zeta)
    &=& (2\pi)^{\frac N2 + \alpha} R^{N-1+\alpha} |\zeta|^\alpha
    \sum_{k=0}^\alpha 2^{-\alpha+k} {\alpha \choose k} (2\pi rR
    |\zeta|)^{-\frac{N-2}2 +k} B_{\alpha, k}(2\pi rR|\zeta|) \\
    &=& 2 \pi^{\alpha+1} r^{-\frac{N-2}2} R^{\frac N2 + \alpha}
    |\zeta|^{-\frac {N-2}2+\alpha} \sum_{k=0}^\alpha {\alpha \choose
      k} |\pi r R \zeta|^{-k} B_{\alpha, k}(2\pi r R |\zeta|),
  \end{IEEEeqnarray*}
  which is \eqref{eq:derisofsphere}. Since $R$ depends
  only on $N$, all the parameters in \eqref{eq:derisofsphere} depend
  only on $N$ and $\alpha$. Since for every half-integer $\nu$,
  $J_\nu(t) = \mO(t^{-\nicefrac12})$ as $t \to \infty$,
  (i) follows for each fixed $r \in \mathopen[1,
  2\mathclose]$, hence by continuity uniformly in $r$.\\
  (ii) \enskip From the series expansion for Bessel functions
  \eqref{eq:bessel}, we get 
  \begin{IEEEeqnarray*}{rCl}
    \tilde m(\zeta)
    &=& 2\pi R^{\frac N2}
    \frac{|\zeta|^{-\frac{N-2}2}}{\Gamma(\nicefrac12)} (\pi R
    |\zeta|)^{\frac{N-2}2} \sum_{k=0}^\infty (-1)^k
    \frac{\Gamma(k+\nicefrac12)}{\Gamma(k + \nicefrac N2)} \frac{(2\pi
      R |\zeta|)^{2k}}{(2k)!} \\
    &=& \sum_{k=0}^\infty (-1)^k \frac{(2\pi R |\zeta|)^{2k}}{(2k)!}
    \prod_{j=0}^{k-1} \frac{j+\nicefrac12}{j+\nicefrac N2}.
    \yesnumber \label{eq:expdinge} 
  \end{IEEEeqnarray*}
  Hence we can apply an argument similar to \eqref{eq:2ndderis} to see
  that
  \begin{equation*}
    |\tilde m(\zeta)| \le e^{-c'(2\pi R |\zeta|)^2}
  \end{equation*}
  holds for some $c'>0$ if $|\zeta|$ is close to
  $0$, say $2 \pi R|\zeta|<\eps$ for some $\eps > 0$. Otherwise, the
  integral representation of Bessel functions gives us 
  \begin{equation}
    \label{eq:derisofsphere2}
    \tilde m(\zeta)
    = \frac{\Gamma(\nicefrac N2)}{\Gamma(\frac{N-1}2) \sqrt \pi}
    \int_{-1}^1 e^{2 \pi i R |\zeta| t} (1- t^2)^{\frac{N-3}2} \dd t.
  \end{equation}
  But we have 
  \begin{IEEEeqnarray*}{rCl}
    \bigg| \int_{-1}^1 e^{2 \pi i R |\zeta| t} (1-
    t^2)^{\frac{N-3}2} \dd t \bigg| 
    &\le& 2 \int_0^1 | e^{2 \pi i R |\zeta| t} (1-
    t^2)^{\frac{N-3}2} | \dd t \\
    &\le& 2 \int_0^1 (1-t^2)^{\frac{N-3}2} \dd t \\
    &=& \frac{\Gamma(\nicefrac12) \Gamma(\frac{N-1}2)}
    {\Gamma(\nicefrac{N}2)},
  \end{IEEEeqnarray*}
  with equality holding if and only if $\zeta = 0$. Combining this
  with \eqref{eq:derisofsphere2} leads to
  \begin{equation*}
    | \tilde m(\zeta) |
    < 1
  \end{equation*}
  for $\zeta \neq 0$, implying $|\tilde m(\zeta)| \le C' < 1$ for
  $\eps \le 2 \pi R |\zeta| \le 2C$. Choosing a positive $c < c'$
  such that $C' \le e^{-c \cdot 4C^2}$ implies
  \begin{equation*}
    |\tilde m(r\zeta)| \le e^{-c(2\pi R |\zeta|)^2}
  \end{equation*}
  for $2\pi R|\zeta| \le C$ and $1 \le r \le 2$. \\
  (iii) \enskip
  Similarly as in (ii), we get
  \begin{equation}
    \label{eq:derisofsphere3}
    \bigg(\frac{\dd}{\dd r}\bigg)^\alpha \tilde m(r\zeta)
    = \frac{\Gamma(\nicefrac N2)}{\Gamma(\frac{N-1}2) \sqrt \pi} (2
    \pi i R |\zeta|)^\alpha \int_{-1}^1 e^{2 \pi i R r |\zeta| t}
    t^\alpha (1- t^2)^{\frac{N-3}2} \dd t.
  \end{equation}
  But we have 
  \begin{IEEEeqnarray*}{rCl}
    \bigg| \int_{-1}^1 e^{2 \pi i R r |\zeta| t} t^\alpha (1-
    t^2)^{\frac{N-3}2} \dd t \bigg|
    &\le& 2 \int_0^1 t^\alpha (1-t^2)^{\frac{N-3}2} \dd t \\
    &=& \frac{\Gamma(\frac{\alpha+1}2) \Gamma(\frac{N-1}2)}
    {\Gamma(\frac{N+\alpha}2)}.
  \end{IEEEeqnarray*}
  Together with \eqref{eq:derisofsphere3}, this gives us
  \begin{equation*}
    \bigg| \bigg(\frac{\dd}{\dd r}\bigg)^\alpha \tilde m(r\zeta)
    \bigg|
    \le (2 \pi R |\zeta|)^\alpha \frac{\Gamma(\frac{\alpha+1}2)
      \Gamma(\nicefrac{N}2)} {\Gamma(\nicefrac12)
      \Gamma(\frac{N+\alpha}2)}.
  \end{equation*}
  Hence (iii) follows if we can estimate
  \begin{equation*}
    \frac{\Gamma(\frac{\alpha+1}2)
      \Gamma(\nicefrac{N}2)} {\Gamma(\nicefrac12)
      \Gamma(\frac{N+\alpha}2)} < 1,
  \end{equation*}
  But this can be established using Stirling's formula: For each $x >
  0$, we have
  \begin{equation*}
    \bigg( \frac{2\pi}x \bigg)^{\nicefrac12} \bigg( \frac xe \bigg)^x
    \le \Gamma(x)
    \le \bigg( \frac{2\pi}x \bigg)^{\nicefrac12} \bigg( \frac xe
    \bigg)^x e^{\frac1{12x}}.
  \end{equation*}
  Thus
  \begin{IEEEeqnarray*}{rCl}
    \frac{\Gamma(\frac{\alpha+1}2) \Gamma(\nicefrac{N}2)}
    {\Gamma(\nicefrac12) \Gamma(\frac{N+\alpha}2)}
    &\le& \frac1{\sqrt \pi} \bigg( \frac{\pi(\alpha +
      N)}{N(\alpha+1)} \bigg)^{\nicefrac12} \bigg( \frac{\alpha+1}{2e}
    \bigg)^{\frac{\alpha+1}2} \bigg( \frac{N}{2e} \bigg)^{\frac{N}2}
    \bigg( \frac{\alpha+N}{2e} \bigg)^{-\frac{\alpha+N}2}
    e^{\frac{1}{6(\alpha+1)}+\frac1{6N}} \\ 
    &\le& \bigg( \frac{\alpha + N}{N(\alpha+1)} \bigg)^{\nicefrac12}
    \bigg( \frac{\alpha+1}{2e} \bigg)^{\frac{\alpha+1}2} \bigg( \frac{N}{2e}
    \bigg)^{\frac{N}2} \bigg( \frac{\alpha+N}{2e}
    \bigg)^{-\frac{\alpha+N}2} e^{\frac1{12}+\frac1{18}} \\
    &=& \frac1{\sqrt 2} \bigg( \frac{\alpha+1}2 \bigg)^{\frac\alpha2}
    \bigg( \frac N2 \bigg)^{\frac{N-1}2} \bigg( \frac{\alpha+N}2
    \bigg)^{-\frac{\alpha+N-1}2} e^{\frac1{12}+\frac1{18}-\frac12} \\
    &<& \frac{e^{-\nicefrac13}}{\sqrt2}.
  \end{IEEEeqnarray*}
  (iv) \enskip By deriving the series expansion \eqref{eq:expdinge} in
  the proof of (ii), we get
  \begin{equation*}
    \bigg(\frac{\dd}{\dd r}\bigg)^\alpha \tilde m(r\zeta)
    = (2 \pi R |\zeta|)^{2\lfloor \frac{\alpha+1}2 \rfloor} \sum_{2k
      \ge \alpha} (-1)^k r^{2k-\alpha} \frac{(2\pi R
      |\zeta|)^{2(k-\lfloor \frac{\alpha+1}2 \rfloor)}}{(2k-\alpha)!}
    \prod_{j=0}^{k-1} \frac{j+\nicefrac12}{j+\nicefrac N2}.
  \end{equation*}
  Since $\alpha > 0$, we use that
  \begin{equation*}
    (2 \pi R |\zeta|)^{2\lfloor \frac{\alpha+1}2 \rfloor} \le (2 \pi R
    |\zeta|)^2,
  \end{equation*}
  since $2 \pi R |\zeta| < 1$. To conclude the proof, we have to
  estimate 
  \begin{equation}
    \label{eq:echtdasbloedeste}
    \sum_{k= \lfloor \frac{\alpha+1}2 \rfloor}^\infty (-1)^k
    r^{2k-\alpha} \frac{t^{2(k-\lfloor \frac{\alpha+1}2
        \rfloor)}}{(2k-\alpha)!} \prod_{j=0}^{k-1}
    \frac{j+\nicefrac12}{j+\nicefrac N2}
    \le C' < 1
  \end{equation}
  for $t = 2 \pi R |\zeta| < 1$. If $\alpha$ is odd, then $\lfloor
  \frac{\alpha+1}2 \rfloor = \frac{\alpha+1}2$, and we get
  \begin{equation}
    \label{eq:nichtnocheine}
    \sum_{k= \frac{\alpha+1}2}^\infty (-1)^k r^{2k-\alpha}
    \frac{t^{2(k- \frac{\alpha+1}2)}}{(2k-\alpha)!}
    \prod_{j=0}^{k-1} \frac{j+\nicefrac12}{j+\nicefrac N2}
    = r \sum_{k=0}^\infty (-1)^{k+\frac{\alpha+1}2}
    \frac{(rt)^{2k}}{(2k+1)!} \prod_{j=0}^{k+\frac{\alpha-1}2} \frac{j
      + \nicefrac12}{j + \nicefrac N2}.
  \end{equation}
  But since $rt < 2$, one can use $(rt)^{2k+2} \le 4(rt)^{2k}$ and
  $(2k+3)! \ge 6\cdot (2k+1)!$ for each $k \in \N$ to see that the
  summands of the alternating right hand side series in
  \eqref{eq:nichtnocheine} are strictly decreasing in $k$. Hence
  \begin{equation*}
    \bigg| r \sum_{k=0}^\infty (-1)^{k+\frac{\alpha+1}2}
    \frac{(rt)^{2k}}{(2k+1)!} \prod_{j=0}^{k+\frac{\alpha-1}2} \frac{j
      + \nicefrac12}{j + \nicefrac N2} \bigg|
    \le r \prod_{j=0}^{\frac{\alpha-1}2} \frac{j + \nicefrac12}{j +
      \nicefrac N2} 
    \le \frac2N
    < 1.
  \end{equation*}
  If $\alpha$ is even, then $\lfloor \frac{\alpha+1}2 \rfloor =
  \frac\alpha2$, and similarly, we get
  \begin{equation}
    \label{eq:nichtnochzweine}
    \sum_{k= \nicefrac{\alpha}2}^\infty (-1)^k r^{2k-\alpha}
    \frac{t^{2k- \alpha}}{(2k-\alpha)!}
    \prod_{j=0}^{k-1} \frac{j+\nicefrac12}{j+\nicefrac N2}
    = \sum_{k=0}^\infty (-1)^{k+\nicefrac{\alpha}2}
    \frac{(rt)^{2k}}{(2k)!}
    {\prod_{j=0}^{k+\nicefrac{\alpha}2-1} \frac{j + \nicefrac12}{j +
        \nicefrac N2}}.
  \end{equation}
  Here, a similar argument provides that since $rt < 2$, the summands
  of the alternating right hand side series of
  \eqref{eq:nichtnochzweine} are decreasing from $k \ge 1$. Thus
  \begin{IEEEeqnarray*}{rCl}
    \bigg| \sum_{k=0}^\infty (-1)^{k+\frac{\alpha}2}
    \frac{(rt)^{2k}}{(2k)!} \prod_{j=0}^{k+\nicefrac{\alpha}2-1}
    \frac{j + \nicefrac12}{j + \nicefrac N2} \bigg|
    &\le& \max\bigg\{ \frac1N + \frac{(rt)^4}{4!} \frac1N \frac3{N+2}
    \frac5{N+4}, \frac{(rt)^2}{2!} \frac1N \frac3{N+2} \bigg\} \\
    &\le& \max\bigg\{\frac53 \frac1N, \frac6{N+2} \frac1N \bigg\} \\
    &=& \frac5{3N}
    < 1.
  \end{IEEEeqnarray*}
\end{proof}
\begin{rem}
  The proofs of (iii) and (iv) even work for $0 < r \le 2$, which will
  be necessary to show the conditions in Lemma \ref{lemma:dgm615}.
\end{rem}
Other than the Leibniz formula \eqref{eq:leibnizder}, we need another
form of $\frac{\dd^j}{\dd r^j} m(r\xi)$ to apply the estimates from
Lemma \ref{lemma:derisofsphere}.
\begin{lemma}
  \label{lemma:altform}
  For each $j \in \N$, and $a = (a_1, \ldots, a_j) \in \ito \ell ^j$,
  define the sequence $\alpha(a) \in \N^\ell$ by
  \begin{equation*}
    \alpha(a)_k = |\{j' : a_{j'} = k\}|.
  \end{equation*}
  Then
  \begin{equation}
    \label{eq:altform}
    \bigg(\frac{\dd}{\dd r}\bigg)^j m(r\xi)
    = \sum_{a \in \ito \ell ^j} \prod_{k=1}^\ell \bigg( \frac{\dd}{\dd
      r} \bigg)^{\alpha(a)_k} \tilde m(r \zeta_k).
  \end{equation}
\end{lemma}
\begin{proof}
  We use induction on $j$, where the case $j=0$ is obvious. So assume
  that \eqref{eq:altform} is true for some $j \in \N$. Then
  \begin{IEEEeqnarray*}{rCl}
    \bigg(\frac{\dd}{\dd r}\bigg)^{j+1} m(r\xi)
    &=& \sum_{a \in \ito \ell ^j} \frac{\dd}{\dd r} \bigg[
    \prod_{k=1}^\ell \bigg( \frac{\dd}{\dd r} \bigg)^{\alpha(a)_k}
    \tilde m(r \zeta_k) \bigg] \\
    &=& \sum_{a \in \ito \ell ^j} \sum_{k=1}^\ell \bigg(
    \frac{\dd}{\dd r} \bigg)^{\alpha(a)_k+1} \tilde m(r \zeta_k)
    \prod_{k' \neq k} \bigg( \frac{\dd}{\dd r} \bigg)^{\alpha(a)_{k'}}
    \tilde m(r \zeta_{k'}) \\
    &=& \sum_{a \in \ito \ell ^j} \sum_{k=1}^\ell \prod_{k'=1}^\ell
    \bigg( \frac{\dd}{\dd r} \bigg)^{\alpha(a)_{k'}+ \delta_{k, k'}}
    \tilde m(r \zeta_{k'}).
  \end{IEEEeqnarray*}
  Since for each $a \in \ito \ell ^j$, $k \in \ito \ell$, the sequence
  $(a, k) \in \ito \ell^{j+1}$ has the property
  \begin{equation*}
    \alpha((a, k))_{k'} = \alpha(a)_{k'} + \delta_{k, k'},
  \end{equation*}
  it follows that
  \begin{equation*}
    \bigg(\frac{\dd}{\dd r}\bigg)^{j+1} m(r\xi)
    = \sum_{a \in \ito \ell ^{j+1}} \prod_{k=1}^\ell \bigg( \frac{\dd}{\dd
      r} \bigg)^{\alpha(a)_k} \tilde m(r \zeta_k).
  \end{equation*}
\end{proof}
With this, we can estimate the derivatives of $m$ straight forward.
\begin{lemma}
  \label{lemma:44}
  Let $K = \lceil \frac{N-1}2 \rceil$. For every $\xi \in \R^n$ and every
  $r \in \mathopen[ 1, 2 \mathclose]$, we have
  \begin{equation}
    \label{eq:deriineed}
    (1+|\xi|)^{\frac{N-1}2-K} \bigg|\bigg(\frac{\dd}{\dd
      r}\bigg)^K m(r\xi)\bigg| \le C_N.
  \end{equation}
\end{lemma}
\begin{proof}
  Fix $0 \neq \xi = (\zeta_1, \ldots, \zeta_\ell)$ and put $t = (t_1,
  \ldots, t_\ell)$, $t_k \ce 2\pi R |\zeta_k|$, $1 \le k \le \ell$. By
  Lemma \ref{lemma:derisofsphere} (i), we can find a constant $A_N$
  such that for $|2 \pi R \zeta| > A_N$ and $0 \le j < K$,
  \begin{equation}
    \label{eq:derisbd12}
    \bigg|\bigg(\frac{\dd}{\dd
      r}\bigg)^j \tilde m(r\zeta)\bigg| \le \frac12.
  \end{equation}
  By the Leibniz formula, we get
  \begin{equation}
    \label{eq:varsplitting}
    \bigg|\bigg(\frac{\dd}{\dd r}\bigg)^K m(r\xi)\bigg|
    = \sum_{j=0}^K {K \choose j} \bigg( \frac{\dd}{\dd r} \bigg)^j
    \bigg[ \prod_{t_k \le A_N} \tilde m(r \zeta_k) \bigg] \cdot \bigg(
    \frac{\dd}{\dd r} \bigg)^{K-j} \bigg[ \prod_{t_k > A_N} \tilde m(r
    \zeta_k) \bigg].
  \end{equation}
  By rearranging the variables, we can assume that $\{ k : t_k
  \le A_N \} = \ito{\ell'}$ for some $\ell' \le \ell$.\\
  Assume further that $K = \nicefrac N2$. In that case, we thus have
  to find bounds
  \begin{equation}
    \label{eq:wassolldas}
    (1 + |\xi|)^{-\nicefrac12} \bigg| \bigg( \frac{\dd}{\dd r}
    \bigg)^j \bigg[ \prod_{k=\ell'+1}^\ell \tilde m(r \zeta_k) \bigg]
    \bigg|
    \le C_N 
  \end{equation}
  and
  \begin{equation}
    \label{eq:wassolldas2}
    \bigg| \bigg( \frac{\dd}{\dd r} \bigg)^j \bigg[
    \prod_{k=1}^{\ell'} \tilde m(r \zeta_k) \bigg] \bigg| \le C_N
  \end{equation}
  for each $j \le K$, with $C_N$ being independent of $\ell$ and
  $\ell'$. As of \eqref{eq:wassolldas2}, we shall assume that also $\{
  k : t_k < 1 \} = \ito{\ell''}$ for some $\ell'' \le \ell'$. Then by
  using the Leibniz formula as in \eqref{eq:varsplitting}, instead of
  \eqref{eq:wassolldas2} we have to find bounds
  \begin{equation}
    \tag{\ref*{eq:wassolldas2}${}^\prime$} \label{eq:wassolldas2pr}
    \bigg| \bigg( \frac{\dd}{\dd r} \bigg)^j \bigg[
    \prod_{k=\ell''+1}^{\ell'} \tilde m(r \zeta_k) \bigg] \bigg| \le
    C_N
  \end{equation}
  and
  \begin{equation}
    \label{eq:wassolldas3}
    \bigg| \bigg( \frac{\dd}{\dd r} \bigg)^j \bigg[
    \prod_{k=1}^{\ell''} \tilde m(r \zeta_k) \bigg] \bigg| \le C_N,
  \end{equation}
  $j \le K$, $C_N$ independent of $\ell$, $\ell'$, and
  $\ell''$. One further assumption will be $\ell'' < \ell' < \ell$,
  since otherwise, at least one of our required estimates would be
  trivial. We will now show \eqref{eq:wassolldas}.
  \begin{IEEEeqnarray*}{rCl}
    \bigg(\frac{\dd}{\dd r}\bigg)^j \bigg[ \prod_{k=\ell'+1}^\ell
    \tilde m(r \zeta_k) \bigg]
    &=& \sum_{\substack{\alpha \in \N^{\ell-\ell'}\\ |\alpha|=j}} {j
      \choose \alpha} \prod_{k=1}^{\ell-\ell'} \bigg(\frac{\dd}{\dd
      r}\bigg)^{\alpha_k} \tilde m(r \zeta_{k+\ell'}) \\
    &=& \sum_{\substack{|\alpha|=j \\ \forall k\, \alpha_k < K}} {j
        \choose \alpha} \prod_{k=1}^{\ell-\ell'} \bigg(\frac{\dd}{\dd
        r} \bigg)^{\alpha_k} \tilde m(r \zeta_{k+\ell'}) \\
    && +\> \sum_{\substack{|\alpha|=j \\ \exists k\, \alpha_k = K}} {j
        \choose \alpha} \prod_{k=1}^{\ell-\ell'} \bigg(\frac{\dd}{\dd
        r} \bigg)^{\alpha_k} \tilde m(r
      \zeta_{k+\ell'}). \yesnumber \label{eq:trick17}
  \end{IEEEeqnarray*}
  Due to the choice of $A_N$, \eqref{eq:derisbd12} yields
  \begin{equation}
    \label{eq:blargh}
    (1+ |\xi|)^{-\nicefrac12} \sum_{\substack{|\alpha|=j \\ \forall
        k\, \alpha_k < K}} {j \choose \alpha} \prod_{k=1}^{\ell-\ell'}
    \bigg| \bigg(\frac{\dd}{\dd r} \bigg)^{\alpha_k} \tilde m(r
    \zeta_{k+\ell'}) \bigg|
    \le \sum_{\substack{|\alpha|=j \\ \forall k\, \alpha_k < K}} {j
      \choose \alpha} \frac1{2^{\ell-\ell'}}
    < \frac{(\ell-\ell')^j}{2^{\ell-\ell'}}
    < C.
  \end{equation}
  The second sum in \eqref{eq:trick17} is equal to $0$ if $j <
  K$. If $j=K$, we use part (i) of Lemma \ref{lemma:derisofsphere} to
  get
  \begin{IEEEeqnarray*}{rCl}
    \IEEEeqnarraymulticol{3}{l}{
      (1 + |\xi|)^{-\nicefrac12} \sum_{\substack{|\alpha|=j \\ \exists
        k\, \alpha_k = K}} {j \choose \alpha} \prod_{k=1}^{\ell-\ell'}
    \bigg| \bigg(\frac{\dd}{\dd r} \bigg)^{\alpha_k} \tilde m(r
    \zeta_{k+\ell'}) \bigg| } \\\quad
    &=& (1 + |\xi|)^{-\nicefrac12} \sum_{k=\ell'+1}^\ell \bigg|
    \bigg(\frac{\dd}{\dd r} \bigg)^j \tilde m(r \zeta_k) \bigg| \cdot
    \prod_{\substack{k' > \ell'\\ k' \neq k}} | \tilde m(r \zeta_{k'})
    | \\
    &\le& \sum_{k=\ell'+1}^\ell (1 + |\xi|)^{-\nicefrac12}
    \frac{\tilde C_N |\zeta_k|^{\nicefrac12}}{2^{\ell-\ell'-1}} \\
    &<& \frac{\tilde C_N(\ell-\ell')}{2^{\ell-\ell'-1}}\\
    &<& C_N. \yesnumber \label{eq:blargh2}
  \end{IEEEeqnarray*}
  Combining \eqref{eq:trick17}, \eqref{eq:blargh}, and
  \eqref{eq:blargh2} gives us \eqref{eq:wassolldas}. \\
  Coming to \eqref{eq:wassolldas2pr}, let $t' = (t_{\ell''+1}, \ldots,
  t_{\ell'})$. By Lemma
  \ref{lemma:derisofsphere} (ii) and (iii), we get
  \begin{IEEEeqnarray*}{rCl}
    \bigg| \bigg( \frac{\dd}{\dd r} \bigg)^j \bigg[
    \prod_{k=\ell''+1}^{\ell'} \tilde m(r \zeta_k) \bigg] \bigg|
    &=& \sum_{\substack{\alpha \in \N^{\ell'-\ell''}\\ |\alpha|=j}} {j
      \choose \alpha} \prod_{k=1}^{\ell'-\ell''} \bigg|
    \bigg(\frac{\dd}{\dd r}\bigg)^{\alpha_k} \tilde m(r
    \zeta_{k+\ell''}) \bigg| \\
    &\le& \sum_{|\alpha|=j} {j \choose \alpha}
    \prod_{k=1}^{\ell'-\ell''} t_{k+\ell''}^{2\alpha_k} e^{-c_N
      t_{k+\ell''}^2} \\
    &=& e^{-c_N |t'|^2} |t'|^{2i}\\
    &<& C_N.
  \end{IEEEeqnarray*}  
  To show \eqref{eq:wassolldas3}, we apply Lemma \ref{lemma:altform}
  to get
  \begin{equation*}
    \bigg(\frac{\dd}{\dd r}\bigg)^j \bigg[
    \prod_{k=1}^{\ell''} \tilde m(r \zeta_k) \bigg]
    = \sum_{a \in \ito{\ell''}^j} \prod_{k=1}^{\ell''} \bigg(
    \frac{\dd}{\dd r} \bigg)^{\alpha(a)_k} \tilde m(r \zeta_k).
  \end{equation*}
  Let $t'' = (t_1, \ldots, t_{\ell''})$. Lemma
  \ref{lemma:derisofsphere} (ii) and (iv) now lead to
  \begin{equation}
    \label{eq:allerletzteimlemma}
    \bigg|\bigg(\frac{\dd}{\dd r}\bigg)^j m(r\xi)\bigg|
    \le \sum_{a \in \ito{\ell''}^j} e^{- c_N |t''|^2}
    \prod_{\alpha(a)_k \neq 0} t_k^2.
  \end{equation}
  To conclude the estimate, we will partition the set $\ito{\ell''}^j$
  as follows. For each $a \in \ito{\ell''}^j$, we have $|\{k :
  \alpha(a)_k \neq 0\}| \le j$. For $i \in \ito j$, put
  \begin{equation*}
    A_i \ce \{ a \in \ito{\ell''}^j : |\{k : \alpha(a)_k \neq 0\}| = i
    \}.
  \end{equation*}
  Then clearly
  \begin{equation*}
    \sum_{a \in A_i} \prod_{\alpha(a)_k \neq 0} t_k^2
    \le \sum_{k_1, \ldots, k_i = 1}^{\ell''} t_{k_1}^2 \cdots t_{k_i}^2.
  \end{equation*}
  Hence by \eqref{eq:allerletzteimlemma}, we get
  \begin{IEEEeqnarray*}{rCl}
    \bigg(\frac{\dd}{\dd r}\bigg)^j \bigg[
    \prod_{k=1}^{\ell''} \tilde m(r \zeta_k) \bigg]
    &\le& e^{- c_N |t''|^2} \sum_{i=1}^j \sum_{a \in A_i}
    \prod_{\alpha(a)_k \neq 0} t_k^2 \\
    &\le& e^{-c_N |t''|^2} \bigg( \sum_{k_1 = 1}^{\ell''} t_{k_1}^2
    + \sum_{k_1, k_2 = 1}^{\ell''} t_{k_1}^2 t_{k_2}^2 + \ldots +
    \sum_{k_1, \ldots, k_j = 1}^{\ell''} t_{k_1}^2 \cdots t_{k_j}^2 \bigg)
    \\
    &=& e^{- c_N |t''|^2} \sum_{k=1}^j |t''|^{2k}\\
    &<& C_N.
  \end{IEEEeqnarray*} 
  This concludes the proof for $K = \nicefrac N2$. The arguments for
  the case $K = \frac{N-1}2$ are basically the same. We also have to
  estimate \eqref{eq:wassolldas2pr} and \eqref{eq:wassolldas3}, but we
  have to drop the factor $(1+|\xi|)^{-\nicefrac12}$ in
  \eqref{eq:wassolldas}. Then the only difference is that the sum
  \begin{equation*}
    \sum_{k=\ell'+1}^\ell (1 + |\xi|)^{-\nicefrac12}
    \frac{\tilde C_N |\zeta_k|^{\nicefrac12}}{2^{\ell-\ell'-1}}
  \end{equation*}
  in \eqref{eq:blargh2} becomes
  \begin{equation*}
    \sum_{k=\ell'+1}^\ell \frac{\tilde C_N }{2^{\ell-\ell'-1}},
  \end{equation*}
  which makes the estimate even slightly easier.
\end{proof}
Now, we are able to show that $\sigma_S - P \dd x$ fulfills
\eqref{eq:erstedgm} and \eqref{eq:zweitedgm}, which allows us to apply
Carbery's interpolation argument.
\begin{lemma}
  \label{lemma:zweitletzt}
  There is a constant $C_N$ such that for every $\theta \in S^{n-1}$
  and every $u \in \R^*$ \em
  \begin{IEEEeqnarray}{rCl}
    | m(u \theta) - \hat P(u \theta) | &\le& C_N \cdot \min\{|u|,
    |u|^{-1}\}, \label{eq:drittedgm} \\
    | \scp{\theta}{\nabla (m -\hat P)(u \theta)} | &\le& C_N \cdot
    \min\{1, |u|^{-1}\}. \label{eq:viertedgm}
  \end{IEEEeqnarray}\em
\end{lemma}
\begin{proof}
  We clearly have
  \begin{equation*}
    |u \cdot \hat P(u\theta)| \le C,\quad |1 - \hat P(u\theta)| \le
    C|u|,\quad \text{and} \quad | \scp{\theta}{\nabla\hat P (u
      \theta)} | \le C \cdot \min\{1, |u|^{-1}\}.
  \end{equation*}
  From \eqref{eq:bourgaindecay}, we also know that
  \begin{equation*}
    |u \cdot m(u\theta)| \le C_N.
  \end{equation*}
  We know that $m$ and $\hat P$ are both bounded. Also,
  from \eqref{eq:expdinge}, a similar argument as in
  \eqref{eq:2ndderis} yields that
  \begin{equation*}
    |m(u \theta)| \ge e^{-c_N u^2}
  \end{equation*}
  for $|u| \le 2$ and some $c_N > 0$ chosen sufficiently small. Hence
  \begin{equation*}
    |1 - m(u\theta)| \le 1 - e^{-c_N u^2} \le C_N' u^2 \le C_N |u|
  \end{equation*}
  if $|u| \le 2$, while $|1 - m(u\theta)| < |u|$ is obvious for $|u| >
  2$. Since
  \begin{equation*}
    |m(u \theta) - \hat P(u \theta)| \le |m(u \theta)| + |\hat P(u
    \theta)| \quad \text{and} \quad |m(u \theta) - \hat P(u \theta)| \le 
    |1-m(u \theta)| + |1 - \hat P(u \theta)|,
  \end{equation*}
  \eqref{eq:drittedgm} follows. For \eqref{eq:viertedgm}, it remains
  to show that
  \begin{equation*}
    | \scp{\theta}{\nabla m(u \theta)} | \le C_N \cdot \min\{1,
    |u|^{-1}\},
  \end{equation*}
  If $|u| \ge 1$, this follows from \eqref{eq:bourgaindecay}, since
  \begin{equation*}
    | u \scp{\theta}{\nabla m(u \theta)} |
    = | \scp{u \theta}{\nabla m(u \theta)} |
    \le C.
  \end{equation*}
  Otherwise, assume that $0 < u < 1$. Write $\theta = (\zeta_1,
  \ldots, \zeta_\ell)$ with $\zeta_k \in \R^N$ as before. Since
  $|\theta| = 1$, we can find $C_N$ such that $2 \pi R |\zeta_k| \le
  C_N$ for each $k$. By Lemma \ref{lemma:derisofsphere} (iii) and
  (iv) and the remark after its proof, we have
  \begin{equation*}
    \big| \frac{\dd}{\dd u} \tilde m(u \zeta_k) \big| < (2\pi R
    |\zeta_k|)^2.
  \end{equation*}
  Hence
  \begin{equation*}
    | \scp{\theta}{\nabla m(u \theta)} |
    = \big| \frac{\dd}{\dd u} m(u \theta) \big|
    = \sum_{k=1}^\ell \big| \frac{\dd}{\dd u} \tilde m(u \zeta_k) \big|
    \prod_{k' \neq k} |\tilde m(u \zeta_{k'}) |
    \le \sum_{k=1}^\ell |2 \pi R \zeta_k|^2
    = (2 \pi R)^2.
  \end{equation*}
  This concludes the estimates.
\end{proof}
This leaves us with finding some $\alpha$ with $\nicefrac1p < \alpha <
1$ such that $m_\alpha$ is an $L^p$-multiplier. We can't interpolate
the corresponding multiplier operators of the family $(m_z)_{0 \le \Re
  z \le \frac{N-1}2}$ directly, but interpolation is still
possible. For this, assume that $\frac{N}{N-1} < p < \frac{N-1}{N-2}$,
and let $\nicefrac1p = 1 - \theta + \nicefrac\theta2$, i.e. $\theta = 2 -
\nicefrac2p$. Then $\frac{N-1}2\theta < 1 < \frac N2 \theta$. Fix
$\eps$ such that $0 < \eps < \frac N2 \theta - 1$, and set
\begin{equation*}
  \alpha \ce \frac{N-1}2\theta - \eps.
\end{equation*}
Then $\alpha > 1 - \nicefrac\theta2 = \nicefrac1p$.
If $N$ is odd, $\frac{N-1}2$ is an integer, and one can use the
formulas \eqref{eq:rieszfd} and \eqref{eq:ekzt} to argue as in
Section 7.3 and Lemma 7.5 of \cite{dgm}, leaving us with bounding
\[
  \bigg( \frac{\dd}{\dd r} \bigg)^j m(r \xi)
\]
for $j \in \oto{\frac{N-1}2-1}$, which we technically did not prove
yet. But this is covered by the estimates \eqref{eq:wassolldas},
\eqref{eq:wassolldas2pr}, and \eqref{eq:wassolldas3} in the proof of
Lemma \ref{lemma:44}. We get that $m_z$ is an
$L^1$-multiplier for $\Re z = -\eps$ and an $L^2$-multiplier for
$\Re z = \frac{N-1}2-\eps$ so that the family
\[
(m_{\frac{N-1}2z-\eps})_{0 \le \Re z \le 1}
\]
is analytic and of admissible growth in the sense of Stein's
interpolation theorem. For $p> \frac{N-1}{N-2}$, we can interpolate
with the endpoint $\infty$. This already suffices to prove Theorem
\ref{theo:one}, since one could simply go from $N$ to $N+1$ if $N$ is
even, but for even $N$, Theorem \ref{theo:twopr} still holds. In that
case, we need to interpolate twice, paying further attention to the
upcoming bounds of the multiplier operators. Suppose we have an
analytic family of operators $(T_z)$ with $0 \le \Re z \le 1$ so that
\begin{equation*}
  \|T_{it}\|_{p_0 \to p_0} \le M_0(t) \quad \text{and} \quad
  \|T_{1+it} \|_{p_1 \to p_1} \le M_1(t),
\end{equation*}
and that we have $b < \pi $ such that
\begin{equation*}
  \sup_{t \in \R} e^{-b|t|} \log M_j(t) < \infty,
\end{equation*}
$j = 0, 1$. Then $\|T_\theta\|_{p\to p} \le M(\theta)$ for
$\frac1p = \frac{1-\theta}{p_0} + \frac\theta{p_1}$ and for instance
from \cite[p.~38]{gra}, it follows that
\begin{equation}
  \label{eq:steinint}
  M(t) = \exp\bigg[ \frac{\sin(\pi t)}2 \int_{\R} \bigg(\frac{\log
    M_0(y)}{\cosh(\pi y) - \cos{\pi t}} + \frac{\log
    M_1(y)}{\cosh(\pi y) + \cos{\pi t}} \bigg) \dd y \bigg].
\end{equation}
For $0 < \eps < 1$, we consider the multipliers $m^\eps_z(\xi) =
(1+|\xi|)^{\frac{N-1}2-\eps-z}m_z(\xi)$, similarly as in
\cite{mue}. For $-\eps \le \Re z \le \frac{N}2$, Lemma 7.4 in
\cite{dgm} gives us the precise bounds
\begin{equation*}
  |m^\eps_{r + it}(\xi)| \le C_N \eps^{-1} (1+t^2)^{\frac{N+1}4}
  e^{\frac{\pi t}2},
\end{equation*}
where $C_N$ depends on the bounds of
\[
  (1+|\xi|)^{-\nicefrac12} \bigg( \frac{\dd}{\dd r} \bigg)^j m(r \xi)
\]
for $j \in \oto{\frac{N-1}2-1}$, which again can be achieved using the
arguments in the proof of Lemma \ref{lemma:44}.
Hence, for every $t \in \R$, the family
$(m^\eps_{\frac{N}2z-\eps+it})_z$ of $L^2$-multipliers has admissible
growth for $0 \le \Re z \le 1$, and by interpolation,
\eqref{eq:steinint} implies
\begin{equation*}
  |m_{\frac{N-1}2-\eps + it}(\xi)| \le e^{C_N |t|}.
\end{equation*}
Thus we also have admissible growth for $\Re z = \frac{N-1}2 - \eps$,
$N$ even, and taking $\eps$ and $\alpha$ as above, this proves Theorem
\ref{theo:twopr}.

%
\bibliographystyle{abbrv}
\bibliography{bib}{}

\begin{thebibliography}{10}

\bibitem{bf}
T.~Bonnesen and W.~Fenchel.
\newblock {\em Theorie der konvexen {K}{\"o}rper}.
\newblock Chelsea Publishing Company, 1948.

\bibitem{bou1}
J.~Bourgain.
\newblock On high dimensional maximal functions associated to convex bodies.
\newblock {\em American Journal of Mathematics}, 108:1467--1476, 1986.

\bibitem{bou3}
J.~Bourgain.
\newblock On the {$L^p$}-bounds for maximal functions associated to convex
  bodies.
\newblock {\em Israel Journal of Mathematics}, 54:257--265, 1986.

\bibitem{bou2}
J.~Bourgain.
\newblock On the {H}ardy-{L}ittlewood maximal function for the cube.
\newblock {\em Israel Journal of Mathematics}, 203(1):275--293, 2014.

\bibitem{car}
A.~Carbery.
\newblock An almost-orthogonality principle with applications to maximal
  functions.
\newblock {\em Bulletin of the American Mathematical Society}, 14(2):269--273,
  1986.

\bibitem{dgm}
L.~Deleaval, O.~Gu{\'e}don, and B.~Maurey.
\newblock Dimension free bounds for the {H}ardy-{L}ittlewood maximal operator
  associated to conves sets, 2016.
\newblock arXiv:1602.02015.

\bibitem{gra}
L.~Grafakos.
\newblock {\em Classical {F}ourier {A}nalysis, 2nd {E}dition}.
\newblock Springer, 2008.

\bibitem{mue}
D.~M{\"u}ller.
\newblock A geometric bound for maximal functions associated to convex bodies.
\newblock {\em Pacific Journal of Mathematics}, 142(2):297--312, 1990.

\bibitem{pis}
G.~Pisier.
\newblock Holomorphic semi-groups and the geometry of banach spaces.
\newblock {\em Annals of Mathematics}, 115:375--392, 1982.

\bibitem{ste}
E.~Stein.
\newblock Three variations on the theme of maximal functions.
\newblock {\em Recent Progress in Fourier Analysis, North-Holland Mathematical
  Studies}, 111:49--56, 1985.

\bibitem{stestr}
E.~Stein and J.~Str{\"o}mberg.
\newblock Behaviour of maximal functions in {$\R^n$} for large {$n$}.
\newblock {\em Arkiv f\"or Matematik}, 21:259--269, 1983.

\end{thebibliography}

\Addresses
\end{document}